\documentclass[12pt]{article}
\usepackage[alphabetic,backrefs]{amsrefs}

 \usepackage{tikz}
 
 \usetikzlibrary{decorations.pathmorphing}
 \tikzstyle{printersafe}=[decoration={snake,amplitude=0pt}]
 
 \usepackage[all]{xy}
 \usepackage{amsmath}
 \usepackage{hyperref}
 \usepackage{amsfonts,graphics,amsthm,amsfonts,amscd,latexsym}
 \usepackage{epsfig}
 \usepackage{graphicx}
 \usepackage{flafter}
 \usepackage{mathtools}
 \usepackage{comment}
 
 \hypersetup{
 	colorlinks=true,    
 	linkcolor=blue,          
 	citecolor=blue,      
 	filecolor=blue,      
 	urlcolor=blue           
 }
 
 \usepackage{tikz}
 \usetikzlibrary{graphs,positioning,arrows,shapes.misc,decorations.pathmorphing}
 
 \tikzset{
 	>=stealth,
 	every picture/.style={thick},
 	graphs/every graph/.style={empty nodes},
 }
 
 \tikzstyle{vertex}=[
 draw,
 circle,
 fill=black,
 inner sep=1pt,
 minimum width=5pt,
 ]
 \usepackage[position=top]{subfig}
 \usepackage{amssymb}
 \usepackage{color}
 \usepackage{tikz-cd}
\usepackage{geometry}
 \usepackage{enumitem}
 
 \usetikzlibrary{decorations.pathmorphing}
 \tikzstyle{printersafe}=[decoration={snake,amplitude=0pt}]
 \DeclareMathOperator{\Ima}{Im}

 \newcommand{\id}{\operatorname{id}}

 \newcommand{\rank}{\operatorname{rank}}

 \newcommand{\pp}{\mathbb{P}}

 \newcommand{\zz}{\mathbb{Z}}

 \newcommand{\Q}{\mathbb{Q}}

 \def\Ox #1.{\mathcal{O}_{#1}}			
 \def\pr #1.{\mathbb P^{#1}}				
 \def\af #1.{\mathbb A^{#1}}			
 \def\ses#1.#2.#3.{0\to #1\to #2\to #3 \to 0}	
 \def\xrar#1.{\xrightarrow{#1}}			
 \def\K#1.{K_{#1}}						
 \def\bA#1.{\mathbf{A}_{#1}}			
 \def\bM#1.{\mathbf{M}_{#1}}				
 \def\bL#1.{\mathbf{L}_{#1}}				
 \def\bB#1.{\mathbf{B}_{#1}}				
 \def\bK#1.{\mathbf{K}_{#1}}			
 \def\subs#1.{_{#1}}					
 \def\sups#1.{^{#1}}

 \DeclareFontFamily{U}{wncy}{}
 \DeclareFontShape{U}{wncy}{m}{n}{<->wncyr10}{}
 \DeclareSymbolFont{mcy}{U}{wncy}{m}{n}
 \DeclareMathSymbol{\Sh}{\mathord}{mcy}{"58}

 \usepackage{tikz}
 \usetikzlibrary{matrix,arrows,decorations.pathmorphing}

 \newtheorem{theorem}{Theorem}[section]
 
 \newtheoremstyle{exampstyle}
 {\topsep} 
 {\topsep} 
 {} 
 {} 
 {\bfseries} 
 {.} 
 {.5em} 
 {} 
 \newtheorem{lemma}[theorem]{Lemma}
 
 \newtheorem{proposition}[theorem]{Proposition}
 \newtheorem{corollary}[theorem]{Corollary}

 \theoremstyle{definition}
 
 \newtheorem{notation}[theorem]{Notation}
 \newtheorem{definition}[theorem]{Definition}

 \newtheorem{remark}[theorem]{Remark}
 
 \theoremstyle{remark}

 \numberwithin{equation}{section}

 \usepackage[autostyle=true]{csquotes} 
 
\title{Computing the Cassels-Tate Pairing for Genus Two Jacobians with Rational Two Torsion Points}

\author{Jiali Yan }
\date{\today}
\begin{document}
\maketitle

\begin{abstract}
	In this paper, we give an explicit formula as well as a practical algorithm for computing the Cassels-Tate pairing on $\text{Sel}^{2}(J) \times \text{Sel}^{2}(J)$ where $J$ is the Jacobian variety of a genus two curve under the assumption that  all points in $J[2]$ are $K$-rational. We also give an explicit formula for the Obstruction map $\text{Ob}: H^1(G_K, J[2]) \rightarrow \text{Br}(K)$ under the same assumption. Finally, we include a worked example demonstrating we can indeed improve the rank bound given by a  2-descent via computing the Cassels-Tate pairing. \\
	\end{abstract}
\section{Introduction}
For any principally polarized  abelian variety $A$ defined over a number field $K$, Cassels and Tate  \cite{cassels1} \cite{cassels2} and \cite{tate} constructed a pairing 
$$\Sh(A) \times \Sh(A) \rightarrow \Q/\zz,$$
that is nondegenerate after quotienting out the maximal divisible subgroup of $\Sh(A)$. This pairing is called the Cassels-Tate pairing and it naturally lifts to a pairing on Selmer groups. One application of this pairing is in improving the bound on the Mordell-Weil rank $r(A)$ obtained by performing a standard descent calculation.  More specifically, if $\Sh(A)$ is finite or if all the $n$-torsion points of $A$ are defined over $K$, the kernel of the Cassels-Tate pairing on $\text{Sel}^n(A) \times \text{Sel}^n(A)$ is equal to  the image of the natural map $\text{Sel}^{n^2}(A) \rightarrow \text{Sel}^n(A)$ induced from the map $A[n^2] \xrightarrow{n} A[n]$, see  \cite[Proposition 1.9.3]{thesis} for details. This shows that carrying out an $n$-descent and computing the Cassels-Tate pairing on $\text{Sel}^n(A) \times \text{Sel}^n(A)$ gives the same rank bound as obtained from $n^2$-descent where $\text{Sel}^{n^2}(A)$ needs to be computed. \\

There have been many results on computing the Cassels-Tate pairing in the case of elliptic curves, such as \cite{cassels98} \cite{steve}  \cite{binary quartic}  \cite{monique} \cite{3 isogeny} \cite{platonic} \cite{3 selmer}. We are interested in the natural problem of  generalizing the different algorithms for computing the Cassels-Tate pairing for elliptic curves to compute the pairing for abelian varieties of higher dimension. \\

In Section 2, we give the preliminary results needed for the later sections, including the homogeneous space definition of the Cassels-Tate pairing. In Section 3, we  
state and prove an explicit formula for the pairing $\langle \;, \; \rangle_{CT}$ on $\text{Sel}^2(J) \times \text{Sel}^2(J)$ where $J$ is the Jacobian variety of a genus two curve under the assumption that  all points in $J[2]$ are $K$-rational. This formula is analogous to that in the elliptic curve case in \cite{cassels98}. In Section 4, we describe a practical algorithm for computing the pairing $\langle \;, \; \rangle_{CT}$  using the formula in Section 3. In section 5, we also give an explicit formula for the Obstruction map $\text{Ob}: H^1(G_K, J[2]) \rightarrow \text{Br}(K)$ under the assumption that all points in $J[2]$ are defined over $K$ generalizing the result in the elliptic curve case \cite [Proposition 3.4]{oneil}, \cite[Theorem 6]{clark}. Finally, in Section 7, we include a worked example demonstrating that  computing the Cassels-Tate pairing can indeed turn a 2-descent to a 4-descent and improve the rank bound given by a 2-descent. The content of this paper is based on  Chapter 4 of the thesis of the author \cite{thesis}.\\

\subsection*{Acknowledgements}
I  express my sincere and deepest gratitute to my PhD supervisor, Dr. Tom Fisher, for his patient guidance and insightful comments  at every stage during my research.\\

\section{Preliminary Results}
\subsection{The set-up}\label{sec:the-set-up}
In this paper, we are working over a number field $K$. For any field $k$, we let $\bar{k}$ denote its algebraic closure and  let $\mu_n \subset \bar{k}$ denote the $n^{th}$ roots of unity in $\bar{k}$.  We let $G_k$ denote the absolute Galois group $\text{Gal}(\bar{k}/k)$.\\

Let  $\mathcal{C}$ be a general \emph{genus two curve} defined over $K$, which  is a smooth projective curve.  It can be given in the following hyperelliptic form:

$$\mathcal{C}: y^2 = f(x)= f_6x^6 +f_5x^5 +f_4x^4 +f_3x^3 +f_2x^2 +f_1x +f_0,$$
where $f_i \in K$, $f_6 \neq 0$ and the discriminant $\triangle(f) \neq 0$, which implies that $f$ has distinct roots in $\bar{K}$.\\

We let $J$ denote the \emph{Jacobian variety} of  $\mathcal{C}$, which is an abelian variety of dimension two defined over $K$ that can be identified with $\text{Pic}^0(\mathcal{C})$. We denote the identity element of $J$ by $\mathcal{O}_J$. Via the natural isomorphism $\text{Pic}^2(\mathcal{C}) \rightarrow \text{Pic}^0(\mathcal{C})$ sending $[P_1+P_2] \mapsto [P_1 + P_2 - \infty^+ - \infty^-]$,  a point $P \in J$ can be identified with an unordered pair of points of $\mathcal{C}$, $\{P_1, P_2\}$. This identification is unique unless $P= \mathcal{O}_J$, in which case it can be represented by any pair of points on $\mathcal{C}$ in the form $\{(x, y), (x, -y)\}$ or $\{\infty^+, \infty^-\}$. Suppose the roots of $f$ are denoted by $\omega_1, ..., \omega_6$. Then $J[2] = \{\mathcal{O}_J, \{(\omega_i, 0), (\omega_j, 0)\}\text{ for } i\neq j\}$. Also, for a point $P \in J$, we let $\tau_P:J \rightarrow J$ denote the translation by $P$ on $J$.\\

As described in \cite[Chapter 3, Section 3]{the book}, suppose $\{P_1, P_2\} $and $\{Q_1, Q_2\}$ represent $P, Q \in J[2]$ where $P_1, P_2, Q_1, Q_2$ are Weierstrass points, then  $$e_2(P, Q)= (-1)^{|\{P_1, P_2\} \cap \{Q_1, Q_2\} |}.$$\\

\subsection{Theta divisor and Kummer surface}\label{sec:theta-divisor-and-kummer-surface}
The \emph{theta divisor}, denoted by $\Theta$, is defined to be the divisor on $J$ that corresponds to the divisor $\{P\}\times \mathcal{C}+ \mathcal{C} \times \{P\}$ on $\mathcal{C} \times \mathcal{C}$ under the birational morphism $\text{Sym}^2\mathcal{C} \rightarrow J$, for some Weierstrass point $P \in \mathcal{C}$. The Jacobian variety $J$ is principally polarized abelian variety via $\lambda: J \rightarrow J^{\vee}$ sending $P$ to $[\tau_P^*\Theta-\Theta].$\\

The \emph{Kummer surface}, denoted by $\mathcal{K}$, is the quotient of $J$ via the involution $[-1]: P \mapsto -P$. The fixed points under the involution are the 16 points of order 2 on $J$ and these map to the 16 nodal singular points of $\mathcal{K}$ (the $\emph{nodes}$). General theory, as in \cite[Theorem 11.1]{abelian varieties}, \cite[page 150]{theta},  shows that the linear system of $n\Theta$ of $J$ has dimension $n^2$. Moreover, $|2\Theta|$ is base point free and $|4\Theta|$ is very ample.\\

\subsection{Explicit embeddings of $J$ and $\mathcal{K}$}\label{sec:explicit-embeddings-of-j-and-mathcalk}
Denote a generic point on the Jacobian $J$ of $\mathcal{C}$ by $\{(x, y), (u, v)\}$. Then, following \cite[Chapter 3, Section 1]{the book}, the morphism from $J$ to $\pp^3$ is given by $$k_1 = 1, k_2 = (x+u), k_3 = xu, k_4 = \beta_0,$$ where
$$\beta_0 = \frac{F_0(x, u)-2yv}{(x-u)^2}$$ with $F_0(x, u)= 2f_0 + f_1(x+u) +2f_2(xu) + f_3(x+u)(xu) + 2f_4(xu)^2 + f_5(x+u)(xu)^2 + 2f_6(xu)^3.$ \\

We denote the above morphism by $J \xrightarrow{|2\Theta|} \mathcal{K} \subset \pp^3$ and it maps $\mathcal{O}_J$ to $(0:0:0:1)$. It is known that its image  in $\mathbb{P}^3$ is precisely the Kummer surface $\mathcal{K}$ and is given by the vanishing of the quartic  $G(k_1, k_2, k_3, k_4)$ with explicit formula given in \cite[Chapter 3, Section 1]{the book}. Therefore, the Kummer surface $\mathcal{K} \subset \pp^3_{k_i}$  is  defined by $G(k_1, k_2, k_3, k_4) =0.$\\

\begin{remark}\label{rem: translation by 2 torsion is linear on K}
	Suppose $P \in J[2]$. We know $\tau_P^*(2\Theta) \sim 2\Theta$ via the polarization. This implies that translation by  $P$ on $J$ induces a linear isomorphism on $\mathcal{K} \subset \pp^3$.\\

\end{remark}

We now look at the  embedding of $J$ in $\pp^{15}$ induced by $|4\Theta|$. Let $k_{ij} = k_ik_j$, for $1 \leq i\leq j \leq 4$. Since $\mathcal{K}$ is irreducible and defined by a polynomial of degree 4, $k_{11}, k_{12}, ..., k_{44}$ are 10 linearly independent even elements in $\mathcal{L}(2\Theta^+ +2 \Theta^-)$. The  six odd basis elements in $\mathcal{L}(2\Theta^+ + 2\Theta^-)$ are given explicitly in   \cite[Section 3]{explicit twist}. A function $g$ on $J$ is \emph{even} when it  is invariant under the involution $[-1]: P \mapsto -P$ and is \emph{odd} when  $g\circ [-1]=-g$. \\

Unless stated otherwise, we will use the basis $k_{11},k_{12}, ..., k_{44},  b_1, ..., b_6$ for $\mathcal{L}(2\Theta^+ + 2\Theta^-)$, to embed $J$ in $\pp^{15}$. The following theorem gives the defining equations of $J$.\\

\begin{theorem}\label{theorem: 72}({\cite[Therorem 1.2]{72 theorem}, \cite[Therorem 1.2]{the gp law paper}}) Let $J$ be the Jacobian variety of the genus two curve $\mathcal{C}$ defined by $y^2=f_6x^6+... +f_1x +  f_0$. The $72$ quadratic forms over $\zz[f_0, . . . , f_6]$ given in  \cite[Appendix A]{72 theorem} are a set of defining equations for the projective variety given by the embedding of $J$ in $\pp^{15}$ induced by the basis of $\mathcal{L}(2\Theta^++2\Theta^-)$ with explicit formulae given in \cite[Definition 1.1]{72 theorem} or  \cite[Definition 1.1]{the gp law paper}. The change of basis between this basis of $\mathcal{L}(2\Theta^++2\Theta^-)$ and $k_{11}, k_{12}, ..., k_{44}, b_1, ..., b_6$ is  given in \cite[Section 3]{explicit twist}.\\
	
\end{theorem}


\subsection{Principal homogeneous space and 2-coverings}\label{sec:principal-homogeneous-space-and-2-coverings}
A \emph{principal homogeneous space} or  \emph{torsor} for $J$ defined over a field $K$ is a variety $V$  together with a morphism $\mu : J \times V \rightarrow V$,  both defined over $K$, that induces a simply transitive action on the $\bar{K}$-points.\\

We say $(V_1, \mu_1)$ and $(V_2, \mu_2)$ are isomorphic over a field extension $K_1$ of $K$ if there is an isomorphism $\phi: V_1 \rightarrow V_2$ defined over $K_1$ that respects the action of $J$.\\

 A \emph{2-covering} of $J$ is a variety $X$ defined over $K$ together with a morphism $\pi : X \rightarrow J$ defined over $K$, such that there exists an isomorphism $\phi : X \rightarrow J$ defined over $\bar{K}$ with $\pi = [2] \circ \phi$. An isomorphism $(X_1, \pi_1) \rightarrow (X_2, \pi_2)$ between two 2-coverings is an isomorphism $h: X_1 \rightarrow X_2$ defined over $K$ with $\pi_1 = \pi_2 \circ h$. We sometimes denote $(X, \pi)$ by $X$ when the context is clear.\\

It can be checked that a 2-covering is a principal homogeneous space. The short exact sequence  $0 \rightarrow J[2] \rightarrow J \xrightarrow{2} J \rightarrow 0$ induces the connecting map in the long exact sequence
\begin{equation}\label{eqn: connecting map}
\delta: J(K) \rightarrow H^1(G_K,J[2]).
\end{equation}
The following two   propositions are  proved in \cite{explicit twist}.\\

\begin{proposition} \label{proposition_2_covering} \cite[Lemma 2.14]{explicit twist}
	Let $(X, \pi)$ be a $2$-covering of an abelian variety $J$ defined over $K$ and choose an isomorphism $\phi:X \rightarrow J$ such that $\pi = [2] \circ \phi$. Then for each  $\sigma \in G_K$, there is a unique point $P \in J[2](\bar{K})$ satisfying $\phi \circ \sigma(\phi^{-1}) = \tau_P$. The map $\sigma \mapsto P$ is a cocycle whose class in $H^1(G_K, J[2])$ does not depend on the choice of $\phi$. This yields a bijection between the set of isomorphism classes of $2$-coverings of $J$ and the set $H^1(G_K,J[2])$.\\
\end{proposition}

\begin{proposition}\label{prop:2 covering has a point}\cite[Proposition 2.15]{explicit twist} Let $X$ be a $2$-covering of $J$ corresponding to the cocycle class $\epsilon \in H^1(G_K, J[2])$. Then $X$ contains a $K$-rational point (equivalently $X$ is a trivial principle homogeneous space) if and only if $\epsilon$ is in the image of the connecting map $\delta$ in \eqref{eqn: connecting map}. \\

\end{proposition}

We also state and prove the following proposition which is useful for the computation of the Cassels-Tate pairing later in Sections \ref{sec:explicit-computation-of-d} and \ref{sec:explicit-computation-of-dp-dq-dr-ds}. A \emph{Brauer-Severi} variety is  a variety that is isomorphic to a projective space over $\bar{K}$.\\

\begin{proposition}\label{prop:BS diagram}
	Let $(X, \pi)$ be a $2$-covering of J, with $\phi \circ [2]= \pi$. Then the linear system $|\phi^*(2\Theta)|$ determines a map $X \rightarrow S$ defined over $K$, where $S$ is a Brauer-Severi variety. Also, there exists an isomorphism $\psi$ defined over $\bar{K}$ making the following diagram commute: 
	\begin{equation}\label{diagram: BS def}
	\begin{tikzcd}
	X  \arrow[r, "|\phi^*(2\Theta)|"] \arrow[d, "\phi"]&S \arrow[d, "\psi"]\\
	J \arrow[r, "|2\Theta|"]& \pp^{3}.
	\end{tikzcd}
	\end{equation}
	
	In particular,  if $(X, \pi)$ corresponds to a Selmer element via the correspondence in Proposition~ \ref{proposition_2_covering}, then the Brauer-Severi variety $S$ is isomorphic to $\pp^{3}$. \\
	
\end{proposition}

\begin{proof}
	Since $(X, \pi)$ is a $2$-covering of $J$, by Proposition \ref{proposition_2_covering}, we have that for each $\sigma \in G_K$, $\phi \circ (\phi^{-1})^{\sigma} = \tau_{P}$ for some $P \in J[2]$. The principal polarization gives $\tau_P^*(2\Theta) \sim 2\Theta$ which implies that $\phi^*(2\Theta) \sim (\phi^{\sigma})^*(2\Theta)$, hence the morphism induced by $|\phi^*(2\Theta)|$ is defined over $K$. \\
	
	Now if $(X, \pi)$ corresponds to a Selmer element, then $X$ everywhere locally has a point by Proposition \ref{prop:2 covering has a point}, and hence $S$ everywhere locally has a point. Since the Hasse principle holds for Brauer-Severi varieties by \cite[Corollary 2.6]{bs hasse}, we know that $S$ has a point over $K$ and hence it is isomorphic to $\pp^{3}$ by \cite[Therem 5.1.3]{QA}. \\
	
\end{proof}

We now make some observations and give some notation.\\

\begin{remark}\label{rem: diagram of twited kummer}
	Let $\epsilon \in \text{Sel}^2(J)$, and let $(J_{\epsilon}, \pi_{\epsilon})$ denote the 2-covering corresponding to $\epsilon$. There exists an isomorphism $\phi_{\epsilon}$ defined over $\bar{K}$ such that $[2] \circ \phi_{\epsilon}=\pi_{\epsilon}$. Then, by Proposition \ref{prop:BS diagram}, we have the following commutative diagram:
	
	\begin{equation}\label{diagram: BS diagram}
	\begin{tikzcd}
	J_{\epsilon} \arrow[r, "|\phi_{\epsilon}^*(2\Theta)|"] \arrow[d, "\phi_{\epsilon}"]&\mathcal{K}_{\epsilon} \subset \pp^3 \arrow[d, "\psi_{\epsilon}"]\\
	J \arrow[r, "|2\Theta|"]& \mathcal{K} \subset \pp^3.
	\end{tikzcd}
	\end{equation}
	The image of $J_{\epsilon}$ under the morphism induced by $|\phi_{\epsilon}^*(2\Theta)|$  is a surface, denoted by $\mathcal{K}_{\epsilon}$, which we call the \emph{twisted Kummer surface} corresponding to $\epsilon$. Also $\psi_{\epsilon}$ is a linear isomorphism $\pp^3 \rightarrow \pp^3$ defined over $\bar{K}$.  \\
	
\end{remark}

\begin{notation}\label{notation: involution}
	Suppose $(J_{\epsilon}, \pi_{\epsilon})$ is the 2-covering of $J$ corresponding to $\epsilon \in H^1(G_K, J[2])$. The involution $[-1]: P \mapsto -P$ on $J$ induces an  involution $\iota_{\epsilon}$ on $J_{\epsilon}$ such that $\phi_{\epsilon} \circ \iota_{\epsilon}=[-1] \circ \phi_{\epsilon}$, where $[2] \circ \phi_{\epsilon}=\pi_{\epsilon}$.  Moreover, the degree 2 morphism $J_{\epsilon} \xrightarrow{|\phi_{\epsilon}^*(2\Theta)|} \mathcal{K}_{\epsilon} \subset \pp^3$ in \eqref{diagram: BS diagram} is precisely the quotient by $\iota_{\epsilon}$ and so an alternative definition of $\mathcal{K}_{\epsilon}$ is as the quotient of $J_{\epsilon}$ by $\iota_{\epsilon}$. We  call a function $g$ on $J_{\epsilon}$ even if it is invariant under $\iota_{\epsilon}$ and odd if $g \circ \iota_{\epsilon}=-g$.\\
	
\end{notation} 

\subsection{Definition of the Cassels-Tate Pairing}\label{sec:definition-of-the-cassels-tate-pairing}

There are four equivalent definitions of the Cassels-Tate pairing stated and proved in \cite{poonen stoll}. In this paper we will only be using the homogeneous space definition of the Cassels-Tate pairing. Suppose $a ,  a' \in \Sh(J)$. Via the polarization $\lambda$, we get $a' \mapsto b$ where $b \in \Sh(J^{\vee}).$ Let $X$ be the (locally trivial) principal homogeneous space defined over $K$ representing $a$. Then $\text{Pic}^0(X_{\bar{K}})$ is canonically isomorphic as a $G_K$-module to $\text{Pic}^0(J_{\bar{K}}) = J^{\vee}(\bar{K}).$ Therefore,  $b \in \Sh(J^{\vee}) \subset H^1(G_K, J^{\vee})$ represents an element in $H^1(G_K, \text{Pic}^0(X_{\bar{K}}))$. \\

Now consider the exact sequence:
$$0 \rightarrow \bar{K}(X)^*/\bar{K}^* \rightarrow \text{Div}^0(X_{\bar{K}}) \rightarrow \text{Pic}^0(X_{\bar{K}}) \rightarrow 0.$$
We can then map $b$ to an element $b' \in H^2(G_K, \bar{K}(X)^*/\bar{K}^*)$ using the long exact sequence associated to the short exact sequence above. Since $H^3(G_K, \bar{K}^*) = 0$, $b'$  has a lift $f' \in H^2(G_K, \bar{K}(X)^*) $ via the long exact sequence induced by the short exact sequence $0 \rightarrow \bar{K}^* \rightarrow \bar{K}(X)^* \rightarrow \bar{K}(X)^*/\bar{K}^* \rightarrow 0:$
\begin{equation}\label{eqn:def}
H^2(G_K, \bar{K}^*) \rightarrow H^2(G_K, \bar{K}(X)^*)\rightarrow H^2(G_K, \bar{K}(X)^*/\bar{K}^*) \rightarrow H^3(G_K, \bar{K}^*)=0.
\end{equation} 
Next we show that $f'_v \in H^2(G_{K_v}, \bar{K_v}(X)^*)$ is the image of an element $c_v \in H^2(G_{K_v}, \bar{K_v}^*).$ This is because $b \in \Sh(J^{\vee})$ is locally trivial which implies its image $b'$ is locally trivial. Then the statement is true by the exactness of local version of sequence \eqref{eqn:def}. \\

We then can define 
$$\langle a, b\rangle = \sum_v \text{inv}_v(c_v) \in \Q/\zz.$$
The Cassels-Tate pairing $\Sh(J) \times \Sh(J) \rightarrow \Q/\zz$ is defined by
$$\langle a, a' \rangle_{CT} := \langle a, \lambda(a')\rangle.$$

We sometimes refer to $\text{inv}_v(c_v)$ above as the local Cassels-Tate pairing between $a, a' \in \Sh(J)$ for a place $v$ of $K$. Note that the local Cassels-Tate pairing depends on the choice of $f'\in H^2(G_K, \bar{K}(X)^*)$. We make the following remarks that are useful for the computation for the Cassels-Tate pairing.\\

\begin{remark}\label{rem:CT_Selmer}
	$\;$\\
	\begin{enumerate}[label=(\roman*)]

		\item  By \cite{poonen stoll}, we know the homogeneous space definition of the Cassels-Tate pairing is independent of all the choices we make.\\

		\item  Via the map $\text{Sel}^2(J) \rightarrow \Sh(J)[2]$, the definition of the Cassels-Tate pairing on $\Sh(J)[2] \times \Sh(J)[2]$ naturally lifts to a pairing on $\text{Sel}^2(J) \times \text{Sel}^2(J)$. In fact, from now on, we will only be considering $\langle \epsilon, \eta \rangle_{CT}$ for $\epsilon, \eta \in \text{Sel}^2(J)$. The principal homogeneous space $X$ in the definition is always taken to be the $2$-covering of $J$ corresponding to $\epsilon$. One can compute $c_v$ by evaluating $f_v'$ at a point in $X(K_v)$ provided that one avoids the zeros and poles of $f_v'$. Note that $X(K_v) \neq \emptyset$  by Proposition \ref{prop:2 covering has a point}.\\
		
	\end{enumerate}
\end{remark}

\subsection{Explicit 2-coverings of $J$}\label{sec:explicit-2-coverings-of-j}
Let $\Omega$ represent the set of 6 roots of $f$, denoted by $\omega_1, ..., \omega_6$. Recall,  as in Proposition \ref{proposition_2_covering}, the isomorphism classes of 2-coverings of $J$ are parameterized by $H^1(G_K, J[2])$. For the explicit computation of the Cassels-Tate pairing, we need the following result on the explicit  2-coverings of $J$ corresponding to elements in $\text{Sel}^2(J)$. We note that this theorem in fact works over any field of characteristic different from 2.\\

\begin{theorem}\cite[Proposition 7.2, Theorem 7.4, Appendix B]{explicit twist}\label{theorem: explicit twist of J} Let $J$ be  the Jacobian variety of a genus two curve defined by $y^2=f(x)$ where $f$ is a degree 6 polynomial and $\epsilon \in \mathrm{Sel}^2(J)$. Embed $J$  in $\pp^{15}$ via the coordinates $k_{11}, k_{12}, ..., k_{44}, b_1, ..., b_6$. There exists $J_{\epsilon} \subset \pp^{15}$ defined over $K$ with Galois invariant coordinates $u_0, ..., u_9, v_1, ..., v_6$ and a linear isomorphism $\phi_{\epsilon}: J_{\epsilon} \rightarrow J$  such that  $(J_{\epsilon}, [2] \circ \phi_{\epsilon})$ is  a $2$-covering of $J$ whose isomorphism class corresponds to the cocycle class $\epsilon$. Moreover,  $\phi_{\epsilon}$ can be explicitly represented by the $16\times 16$ matrix $R=
	\begin{bmatrix}
	R_1&0\\
	0&R_2\\
	\end{bmatrix}$ for some $10 \times 10$ matrix $R_1$ and some $6 \times 6$ matrix $R_2$.\\

\end{theorem}

\begin{remark}\label{rem: explicit twist of J formula}
	The explicit formula for $\phi_{\epsilon}$ is given in the beginning of \cite[Section 7]{explicit twist} and depends only on $\epsilon$ and the underlying genus two curve. Note that the coordinates $u_0, ..., u_9, v_1, ..., v_6$ are derived from another set of coordinates $c_0, ..., c_9, d_1, ..., d_6$ defined in \cite[Definitions 6.9, 6.11]{explicit twist} where $c_0, ..., c_9$ are even and $d_1, ..., d_6$ are odd. This set of coordinates are in general not  Galois invariant, however, they are in the case where all points of $J[2]$ are defined over the base field.\\
	\end{remark}

\section{Formula for the Cassels-Tate Pairing }\label{sec:formula-for-the-cassels-tate-pairing}
From now on, we always assume that the  genus two curve $\mathcal{C}$   is defined by $y^2=f(x)$ such that all roots of $f$ are defined over $K$. Note that this implies that all points in $J[2]$ are defined over $K$ which is equivalent to all the Weierstrass points defined over $K$. In this section, under the above assumption, we state and prove an explicit formula for the Cassels-Tate pairing on $\text{Sel}^{2}(J) \times \text{Sel}^2(J)$. \\

Let the genus two curve $\mathcal{C}$ be of the form 
$$\mathcal{C}: y^2 = \lambda(x-\omega_1)(x-\omega_2)(x-\omega_3)(x-\omega_4)(x-\omega_5)(x-\omega_6), $$
where $\lambda, \omega_i \in K$ and $\lambda \neq 0$.  Its Jacobian variety is denoted by $J$.\\

The 2-torsion subgroup $J[2]$ has basis  
\begin{align*}
P =\{(\omega_1, 0), (\omega_2, 0)\}, & \;\;\;Q =\{(\omega_1, 0), (\omega_3, 0)\},\\
R =\{(\omega_4, 0), (\omega_5, 0)\}, & \;\;\;S=\{(\omega_4, 0), (\omega_6, 0)\}.\\
\end{align*} 
By the  discussion at the end of Section \ref{sec:the-set-up}, the Weil pairing is given relative to this basis by:
\begin{equation}\label{equation: WP matrix}
W=\begin{bmatrix}
1&-1&1&1\\
-1&1&1&1\\
1&1&1&-1\\
1&1&-1&1\\
\end{bmatrix}.
\end{equation}
More explicitly,  $W_{ij}$ denotes the Weil pairing between the $i^{th}$ and $j^{th}$ generators. \\

We now show that this choice of basis determines an isomorphism $H^1(G_K, J[2]) \cong (K^*/(K^*)^2)^4$. Consider the map $ J[2] \xrightarrow{w_2} (\mu_2(\bar{K}))^4, $ where  $w_2$ denotes taking the  Weil pairing with $P, Q, R, S$.  Since $P, Q, R, S$ form a basis for $J[2]$ and the Weil pairing is a nondegerate bilinear pairing, we get that $w_2$ is injective. This implies that $w_2$ is an isomorphism as $|J[2]|= |(\mu_2(\bar{K}))^4|=16$. We then get

\begin{equation}\label{eq: isomorphism of H^1(J[2])}
H^1(G_K, J[2]) \xrightarrow{w_{2, *}}H^1(G_K, (\mu_2(\bar{K}))^4) \cong (K^*/(K^*)^2)^4,
\end{equation}
where $w_{2, *}$ is induced by $w_2$ and $\cong$ is the Kummer isomorphism derived from Hilbert's Theorem 90. Since  the map \eqref{eq: isomorphism of H^1(J[2])} is an isomorphism, we can represent elements in $H^1(G_K, J[2])$ by elements in $(K^*/(K^*)^2)^4$.\\

Before stating and proving the formula for the Cassels-Tate pairing, we first state and prove the following lemma. \\

\begin{lemma}\label{lem: existence of rational divisor D_T}
	For $\epsilon \in \mathrm{Sel}^2(J)$, let $(J_{\epsilon}, \pi_{\epsilon})$  denote the corresponding $2$-covering of $J$. Hence, there exists an isomorphism $\phi_{\epsilon}: J_{\epsilon} \rightarrow J$ defined over $\bar{K}$ such that $[2] \circ \phi_{\epsilon}= \pi_{\epsilon}$. Suppose $T \in J(K)$ and $T_1 \in J(\bar{K})$ satisfy $2T_1=T$. Then 
	\begin{enumerate}[label=(\roman*)]
		\item There exists a $K$-rational divisor $D_T$ on $J_{\epsilon}$ which represents the divisor class of $\phi_{\epsilon}^*(\tau_{T_1}^*(2\Theta))$.
		\item Let $D$ and $D_T$ be   $K$-rational divisors on $J_{\epsilon}$ representing the divisor class of $\phi_{\epsilon}^*(2\Theta)$ and $\phi_{\epsilon}^*(\tau_{T_1}^*(2\Theta))$ respectively. Then $D_T-D \sim \phi_{\epsilon}^*(\tau_T^*\Theta-\Theta)$. Suppose $T$ is a two torsion point. Then $2D_T -2D$ is a $K$-rational principal divisor. Hence, there exists a $K$-rational function $f_T$ on $J_{\epsilon}$ such that $\text{div}(f_T)= 2D_T -2D$. \\
		
	\end{enumerate}
\end{lemma}

\begin{proof}
	By definition of a 2-covering, $[2] \circ \phi_{\epsilon}= \pi_{\epsilon}$ is a morphism defined over $K$. Also, by Proposition \ref{proposition_2_covering},  $\phi_{\epsilon} \circ (\phi_{\epsilon}^{-1})^{\sigma} = \tau_{\epsilon_{\sigma}}$ for all $\sigma \in G_K$, where $(\sigma \mapsto \epsilon_{\sigma})$ is a cocycle representing $\epsilon$. Since $[2] \circ \tau_{T_1} \circ \phi_{\epsilon}=  \tau_T \circ [2]\circ \phi_{\epsilon}=\tau_T \circ \pi_{\epsilon}$ and  $\tau_T$ is defined over $K$, $(J_{\epsilon}, \tau_T \circ \pi_{\epsilon})$ is also a 2-covering of $J$. We compute $\tau_{T_1} \circ \phi_{\epsilon} \circ ((\tau_{T_1} \circ \phi_{\epsilon} )^{-1})^{\sigma }= \tau_{T_1} \circ \phi_{\epsilon} \circ (\phi_{\epsilon}^{-1})^{\sigma} \circ \tau_{-\sigma(T_1)}=\tau_{\epsilon_{\sigma}} \circ \tau_{T_1} \circ \tau_{-\sigma(T_1)},$ for all $\sigma \in G_K$. This implies the 2-covering $(J_{\epsilon}, \tau_T \circ \pi_{\epsilon})$  corresponds to the element in $ H^1(G_K, J[2])$ that is represented  by the cocycle $(\sigma \mapsto \epsilon_{\sigma} + T_1 -\sigma(T_1))$. Hence,  $(J_{\epsilon}, \tau_T \circ \pi_{\epsilon})$ is the 2-covering of $J$ corresponding to $\epsilon+ \delta(T)$, where  $\delta$ is the connecting map as in \eqref{eqn: connecting map}. By Proposition \ref{prop:BS diagram}, there exists a commutative diagram:
	
	\[\begin{tikzcd}[column sep=large]
	J_{\epsilon} \arrow[r, "|\phi_{\epsilon}^*(\tau_{T_1}^*(2\Theta))|"] \arrow[d, "\tau_{T_1} \circ \phi_{\epsilon}"]&\pp^3 \arrow[d, "\psi_{\epsilon}"]\\
	J \arrow[r, "|2\Theta|"]& \pp^3,
	\end{tikzcd}
	\]
	where the morphism $J_{\epsilon} \xrightarrow{|\phi_{\epsilon}^*(\tau_{T_1}^*(2\Theta))|}\pp^3$ is defined over $K$. So the pull back of a  hyperplane section via this morphism gives us a rational divisor $D_T$ representing the divisor class of $\phi_{\epsilon}^*(\tau_{T_1}^*(2\Theta))$ as required by (i).\\
	
	Since the polarization $\lambda: J \rightarrow J^{\vee}$ is an isomorphism and $2T_1=T$, we have $\phi_{\epsilon}^*(\lambda(T)) = [\phi_{\epsilon}^*(\tau_T^*\Theta- \Theta)] = [\phi_{\epsilon}^*(\tau_{T_1}^*(2\Theta))] - [\phi_{\epsilon}^*(2\Theta)]=[D_T]-[D]$.  The fact that $T$ is a two torsion point implies that $2\phi_{\epsilon}^*(\lambda(P))= 0$. Hence, $2D_T-2D$ is a $K$-rational principal divisor which gives (ii).\\

\end{proof}

The following remark explains how we will use Lemma \ref{lem: existence of rational divisor D_T}  in the formula for the Cassels-Tate pairing on $\text{Sel}^2(J) \times \text{Sel}^2(J)$.\\

\begin{remark}\label{rem: DP, DQ, DR, DS}
	Applying Lemma  \ref{lem: existence of rational divisor D_T}(i) with $T = \mathcal{O}_J, P, Q, R, S \in J[2]$ gives divisors $D=D_{\mathcal{O}_J}$ and $D_P, D_Q, D_R D_S$. Then by Lemma  \ref{lem: existence of rational divisor D_T}(ii), there exist $K$-rational functions $f_P, f_Q, f_R, f_S$ on $J_{\epsilon}$ such that $\text{div}(f_T)= 2D_T -2D$  for $T=P, Q, R, S$.\\
	
	
	
\end{remark}
\begin{theorem}\label{thm: 1}
	Let $J$ be the Jacobian variety of a genus two curve $\mathcal{C}$ defined over a number field $K$ where all points in $J[2]$ are defined over $K$.	For any $\epsilon, \eta \in \mathrm{Sel}^2(J)$, let $(J_{\epsilon}, [2]\circ \phi_{\epsilon})$ be the $2$-covering of $J$ corresponding to $\epsilon$ where $\phi_{\epsilon}: J_{\epsilon} \rightarrow J$ is an isomorphism defined over $\bar{K}$. Fix a choice of basis $P, Q, R, S$ for $J[2]$, with the Weil pairing given by matrix \eqref{equation: WP matrix}. Let $(a, b, c, d)$ denote the image of $\eta$ via  $H^1(G_K, J[2]) \cong(K^*/(K^*)^2)^4$, where this is the isomorphism  induced by taking the Weil pairing with $P, Q, R, S$.  Let $f_P, f_Q, f_R, f_S$ be the $K$-rational functions on $J_{\epsilon}$ defined in Remark \ref{rem: DP, DQ, DR, DS}. Then the  Cassels-Tate pairing $\langle\;,\; \rangle_{CT}: \mathrm{Sel}^2(J) \times \mathrm{Sel}^2(J) \rightarrow \{\pm 1\}$ is given by
	$$\langle \epsilon, \eta\rangle_{CT} = \prod_{\text{place } v}(f_P(P_v), b)_v(f_Q(P_v), a)_v(f_R(P_v), d)_v(f_S(P_v), c)_v,$$
	where $(\;,\;)_v$ denotes the Hilbert symbol for a given place $v$ of $K$ and  $P_v$ is an arbitrary choice of a local point on $J_{\epsilon}$ avoiding the zeros and poles of $f_P, f_Q, f_R, f_S$.\\

\end{theorem}

\begin{proof}

	We know $\eta \in H^1(G_K, J[2])$ corresponds to $(a, b, c, d) \in (K^*/(K^*)^2)^4$ via taking the Weil pairing with $P, Q, R, S$. Hence, $\eta$ is represented by the cocycle $$\sigma \mapsto \tilde{b}_{\sigma}P + \tilde{a}_{\sigma}Q+ \tilde{d}_{\sigma}R+ \tilde{c}_{\sigma}S,$$
	where $\sigma \in G_K$ and for each element $x \in K^*/(K^*)^2$, we define $\tilde{x}_{\sigma} \in \{0, 1\}$ such that $(-1)^{\tilde{x}_{\sigma}}=\sigma(\sqrt{x})/\sqrt{x}$. \\
	
	Then the  image of $\eta$ in $H^1(G_K, \text{Pic}^0(J_{\epsilon}))$ is represented by the cocycle that sends $\sigma \in G_K$ to 
	$$\tilde{b}_{\sigma} \phi_{\epsilon}^*[\tau_P^*\Theta-\Theta] + \tilde{a}_{\sigma}\phi_{\epsilon}^*[\tau_Q^*\Theta-\Theta]+ \tilde{d}_{\sigma}\phi_{\epsilon}^*[\tau_R^*\Theta-\Theta]+ \tilde{c}_{\sigma}\phi_{\epsilon}^*[\tau_S^*\Theta-\Theta].$$

	By Remark \ref{rem: DP, DQ, DR, DS}, there exist $K$-rational divisors $D_P, D_Q, D_R, D_S$ on $J_{\epsilon}$ such that the above cocycle sends $\sigma \in G_K$ to 
	$$\tilde{b}_{\sigma} [D_P-D] + \tilde{a}_{\sigma}[D_Q-D] + \tilde{d}_{\sigma}[D_R-D]+ \tilde{c}_{\sigma}[D_S-D].$$
	
	We need to map this element in $H^1(G_K, \text{Pic}^0(J_{\epsilon}))$ to an element in $H^2(G_K, \bar{K}(J_{\epsilon})^*/\bar{K}^*)$ via the connecting map induced by the short exact sequence $$0 \rightarrow \bar{K}(J_{\epsilon})^*/\bar{K}^*  \rightarrow \text{Div}^0(J_{\epsilon}) \rightarrow \text{Pic}^0(J_{\epsilon}) \rightarrow 0.$$ Hence, by the formula for the connecting map   and the fact that the divisors $D, D_P, D_Q, D_R,$ $ D_S$ are all $K$-rational, we get that the corresponding element in $H^2(G_K, \bar{K}(J_{\epsilon})^*/\bar{K}^*)$ has image in $H^2(G_K, \text{Div}^0(J_{\epsilon}))$ represented by the following cocycle:
	\begin{align*}
	(\sigma,  \tau) \mapsto & (\tilde{b}_{\tau}-\tilde{b}_{\sigma\tau}+ \tilde{b}_{\sigma}) (D_P-D) + (\tilde{a}_{\tau}-\tilde{a}_{\sigma\tau}+ \tilde{a}_{\sigma})(D_Q-D) \\
	&+ (\tilde{d}_{\tau}-\tilde{d}_{\sigma\tau}+ \tilde{d}_{\sigma})(D_R-D)+(\tilde{c}_{\tau}-\tilde{c}_{\sigma\tau}+ \tilde{c}_{\sigma})(D_S-D),
	\end{align*}
	for  $\sigma,  \tau \in G_K$.\\
	
	It can be checked that,  for $x \in K^*/(K^*)^2$ and $\sigma, \tau \in G_K$, we get $\tilde{x}_{\tau}-\tilde{x}_{\sigma\tau}+ \tilde{x}_{\sigma}=2$ if both $\sigma$ and $\tau$ flip $\sqrt{x}$ and otherwise it is equal to zero. Define $\iota_{\sigma, \tau, x} = 1$ if both $\sigma$ and $\tau$ flip $\sqrt{x}$ and otherwise $\iota_{\sigma, \tau, x} = 0$. Note that the map that sends $x \in K^*/(K^*)^2$ to the  class of $(\sigma, \tau) \mapsto \iota_{\sigma, \tau, x}$ explicitly realizes the map $K^*/(K^*)^2 \cong H^1(G_K, \frac{1}{2}\zz/\zz) \subset H^1(G_K, \Q/\zz) \rightarrow H^2(G_K, \zz)$. Then, for $\sigma, \tau \in G_K$, the cocycle in the last paragraph sends $(\sigma, \tau)$ to 
	$$
	\iota_{\sigma, \tau, b} \cdot 2(D_P-D) + \iota_{\sigma, \tau, a} \cdot 2(D_Q-D) +  \iota_{\sigma, \tau, d} \cdot 2(D_R-D)+ \iota_{\sigma, \tau, c} \cdot 2(D_S-D).
	$$
	
	Hence, by Remark \ref{rem: DP, DQ, DR, DS}, the corresponding element in $H^2(G_K, \bar{K}(J_{\epsilon})^*/\bar{K}^*)$ is represented by
	$$(\sigma, \tau) \mapsto [f_P^{\iota_{\sigma, \tau, b}} \cdot f_Q^{\iota_{\sigma, \tau, a}} \cdot f_R^{\iota_{\sigma, \tau, d}} \cdot f_S^{\iota_{\sigma, \tau, c}}],$$
	for all $\sigma, \tau \in G_K$.\\
	
	For each place $v$ of $K$, following the homogeneous space definition of $\langle \epsilon, \eta \rangle_{CT}$ as given in Section \ref{sec:definition-of-the-cassels-tate-pairing}, we obtain an element in $H^2(G_{K_v}, \bar{K_v^*})$ from the long exact sequence  induced by the short exact sequence $0 \rightarrow \bar{K_v^*} \rightarrow \bar{K_v}(J_{\epsilon})^* \rightarrow \bar{K_v}(J_{\epsilon})^*/\bar{K_v}^* \rightarrow 0$. The long exact sequence is the local version of \eqref{eqn:def} with $X$ replaced by $J_{\epsilon}$. By Remark \ref{rem:CT_Selmer}(ii), this element in $H^2(G_{K_v}, \bar{K_v}^*)$ can be represented by\\
	$$(\sigma, \tau) \mapsto f_P(P_v)^{\iota_{\sigma, \tau, b}} \cdot f_Q(P_v)^{\iota_{\sigma, \tau, a}} \cdot f_R(P_v)^{\iota_{\sigma, \tau, d}} \cdot f_S(P_v)^{\iota_{\sigma, \tau, c}}, $$
	for all $\sigma, \tau \in G_K$ and some local point $P_v \in J_{\epsilon}(K_v)$ avoiding the zeros and poles of $f_P, f_Q, f_R, f_S$. \\
	
	Hence, the above element in $\text{Br}(K_v) \cong H^2(G_{K_v}, \bar{K_v}^*)$ is the class of the  tensor product of quaternion algebras
	$$(f_P(P_v), b) +(f_Q(P_v), a) + (f_R(P_v), d)+ (f_S(P_v), c).$$
	Then, we have that
	\begin{align*}
	\text{inv}&\big((f_P(P_v), b) +(f_Q(P_v), a) + (f_R(P_v), d)+ (f_S(P_v), c)\big)\\
	&= (f_P(P_v, b)_v (f_Q(P_v), a)_v(f_R(P_v), d)_v (f_S(P_v), c)_v,
	\end{align*}
	where $(\;,\; )_v$ denotes the Hilbert symbol: $K_v^* \times K_v^* \rightarrow \{1, -1\}$, as required.\\
	\end{proof}

\begin{remark}\label{rem: CTP finite product}
	In  Section \ref{sec:prime-bound}, we will directly show that the formula for the Cassels-Tate pairing on $\text{Sel}^2(J) \times \text{Sel}^2(J)$ given in Theorem \ref{thm: 1} is a finite product.\\
	
\end{remark}

\section{Explicit Computation}
In this section, we explain how we explicitly compute the Cassels-Tate pairing on $\text{Sel}^2(J) \times \text{Sel}^2(J)$ using the formula given in Theorem \ref{thm: 1}, under the assumption that all points in $J[2]$ are defined over $K$. We fix $\epsilon \in \text{Sel}^2(J)$ and $(J_{\epsilon}, [2] \circ \phi_{\epsilon})$, the 2-covering of $J$ corresponding to $\epsilon$ with $\phi_{\epsilon}: J_{\epsilon} \subset \pp^{15}\rightarrow J \subset \pp^{15}$ given in Theorem \ref{theorem: explicit twist of J}. The statement of Theorem \ref{thm: 1} suggests that we need to compute the $K$-rational divisors $D, D_P, D_Q, D_R, D_S$ on $J_{\epsilon}$ and the $K$-rational function $f_P, f_Q, f_R, f_S$ on $J_{\epsilon}$, as in Remark \ref{rem: DP, DQ, DR, DS}.\\

\subsection{Computing the twist of  the Kummer surface}\label{sec:computing-the-twist-of--the-kummer-surface}
We describe a practical method for computing a linear isomorphism $\psi_{\epsilon}: \mathcal{K}_{\epsilon} \subset \pp^3 \rightarrow \mathcal{K} \subset \pp^3$ corresponding to $\epsilon$. More explicitly, we need to compute $\psi_{\epsilon}$ such that $\psi_{\epsilon}(\psi_{\epsilon}^{-1})^{\sigma}$ is the action of translation by $\epsilon_{\sigma} \in J[2]$ on $\mathcal{K}$ and $(\sigma \mapsto \epsilon_{\sigma})$ is a cocycle representing $\epsilon$. Since all points in $J[2]$ are defined over $K$, the coboundaries in $B^1(G_K, J[2])$ are trivial. Therefore,  these conditions determine $\psi_{\epsilon}$ uniquely up to a change of linear automorphism of $\mathcal{K}_{\epsilon} \subset \pp^3$ over $K$. \\

For each $T \in J[2]$, we have an explicit formula for $M_T \in \text{GL}_4(K)$, given in \cite[Chapter 3 Section 2]{the book}, representing the action of translation by $T\in J[2]$ on the Kummer surface $\mathcal{K} \subset \pp^3$. It can be checked that they  form  a basis of $\text{Mat}_4(K)$ and we suppose $M_P^2 = c_PI, M_Q^2 = c_QI,M_R^2 = c_RI, \; \text{and} \; M_S^2 = c_SI.$
The explicit formulae for $c_P,c_Q,c_R,c_S$ can also be found in \cite[Chapter 3 Section 2]{the book}. Moreover, by \cite[Chapter 3 Section 3]{the book} and the Weil pairing relationship among the generators $P, Q, R, S$ of $J[2]$ specified by \eqref{equation: WP matrix}, we know that $[M_P, M_Q] = [M_R, M_S] = -I$ and the commutators of the other pairs are trivial.\\


Suppose $(a, b, c, d) \in (K^*/(K^*)^2)^4$ represents $\epsilon$. Let $A \in \text{GL}_4(\bar{K})$ represent the linear isomorphism $\psi_{\epsilon}$ and let $M_T'=A^{-1}M_TA \in \text{GL}_4(\bar{K})$ represent the action of $T$ on the twisted Kummer $\mathcal{K}_{\epsilon}$. It can be checked, see \cite[Lemma 3.2.1]{thesis}  for details, that  the set of matrices in $\text{PGL}_4(\bar{K})$ that commute with $M_T$ in $\text{PGL}_4(\bar{K})$ for all $T \in J[2]$ is $\{[M_T], T \in J[2]\}$. This implies that any $B \in \text{GL}_4(\bar{K})$ such that $[M_T]'=[B^{-1}M_TB] \in \text{PGL}_4(\bar{K})$ for any $T \in J[2]$ is equal to a multiple of $M_T$ composed with $A$ and so can also represents $\psi_{\epsilon}$. Hence, it will  suffice to compute the matrices $M_T'$.\\

Consider $[M_T'] \in \text{PGL}_4(\bar{K})$ and $\sigma \in G_K$. We have $$[M_T']([M_T']^{-1})^{\sigma}=[A^{-1}M_TA(A^{-1})^{\sigma}M_T^{-1}A^{\sigma}]\in \text{PGL}_4(\bar{K}).$$ Recall that for each element $x \in K^*/(K^*)^2$, we define $\tilde{x}_{\sigma} \in \{0, 1\}$ such that $(-1)^{\tilde{x}_{\sigma}}=\sigma(\sqrt{x})/\sqrt{x}$. Since $[A(A^{-1})^{\sigma}]=[M_P^{\tilde{b}_{\sigma}}M_Q^{\tilde{a}_{\sigma}}M_R^{\tilde{d}_{\sigma}}M_S^{\tilde{c}_{\sigma}}]$, we have $[M_T']$ is in $\text{PGL}_4(K)$.  This means that we can redefine $M_T'=\lambda_TA^{-1}M_TA$ for some  $\lambda_T\in \bar{K}$ such that $M_T'\in \text{GL}_4(K)$ by Hilbert Theorem~90. \\

Let $N_P=1/\sqrt{c_P}M_P, N_Q=1/\sqrt{c_Q}M_Q, N_R=1/\sqrt{c_R}M_R, N_S=1/\sqrt{c_S}M_S$. Then $N_T^2=I$ for $T=P, Q, R, S$.  Define $N_T'=A^{-1}N_TA\in \text{GL}_4(\bar{K})$ for $T=P, Q, R, S$. We note that $N_T', M_T'$ represent the same element in $\text{PGL}_4(K)$ and $N_T'^2=I$ for each $T=P, Q, R, S$. Suppose $M_P'^2=\alpha_PI, M_Q'^2=\alpha_QI, M_R'^2=\alpha_RI, M_S'^2=\alpha_SI.$ Then $N_T'=1/\sqrt{\alpha_T}M_T'$ for each $T=P, Q, R, S$. Note that there might be some sign issues here but they will not affect the later computation. Since 
$$N_P'(N_P'^{-1})^{\sigma}=A^{-1}N_PA(A^{-1})^{\sigma}(N_P^{-1})^{\sigma}A^{\sigma},$$
using $N_P=1/\sqrt{c_P}M_P$ and $N_P'=1/\sqrt{\alpha_P}M_P'$ with $M_P, M_P' \in \text{GL}_4(K)$, we compute that $$\frac{\sigma(\sqrt{\alpha_P})}{\sqrt{\alpha_P}}=\frac{\sigma(\sqrt{a})}{\sqrt{a}}\frac{\sigma(\sqrt{c_P})}{\sqrt{c_P}},$$
and similar equations for $Q, R, S.$\\

This implies that $\alpha_P=c_P  a$ up to squares in $K$ and so via rescaling $M_P'$ by elements in $K$, we have  $M_P'^2 = c_P  aI$. Similarly, $M_Q'^2 = c_Q bI, M_R'^2 = c_R cI, M_S'^2 = c_S  dI$.  We note that we also have  $[M_P', M_Q'] = [M_R', M_S'] = -I$ and the commutators of the other pairs are trivial. This implies that 

$$\text{Mat}_4(K) \cong (c_P  a, c_Q  b) \otimes (c_R  c, c_S  d)$$
$$M_P' \mapsto i_1 \otimes 1, M_Q' \mapsto j_1 \otimes 1, M_R' \mapsto 1 \otimes i_2, M_S' \mapsto 1\otimes j_2,$$
where $(c_P a, c_Q  b)$ and $(c_R  c,  c_S d)$ are quaternion algebras with generators  $i_1, j_1$ and $i_2, j_2$ respectively. In Section5, we will interpret this isomorphism as saying that the image of $\epsilon$ via the obstruction map is trivial.\\

Let $A = (c_P  a, c_Q b), B= (c_R  c, c_S d)$. By the argument above, we know $A \otimes B$ represents the trivial element in $\text{Br}(K)$ and an explicit isomorphism  $A \otimes B \cong \text{Mat}_4(K)$ will give us the explicit matrices $M_P', M_Q', M_R', M_S'$ we seek. Since the classes of $A, B$ are in  $\text{Br}[2]$, we have $A, B$ representing the same element in  $\text{Br}(K)$. This implies that $A \cong B$ over $K$, by Wedderburn's Theorem. We have the following lemma.\\

\begin{lemma}\label{lem:  computation of M_P' etc}
	Consider a tensor product of two quaternion algebras $A \otimes B$, where $A =(\alpha, \beta), \; B=(\gamma, \delta),$ with generators $i_1, j_1$ and $i_2, j_2$ respectively. Suppose there is an isomorphism  $\psi: B \xrightarrow{\sim} A$ given by
	\begin{align*} 
	i_2 \mapsto a_1 \cdot 1 + b_1 \cdot i_1 + c_1 \cdot j_1 + d_1 \cdot i_1j_1,\\
	j_2 \mapsto a_2 \cdot 1 + b_2 \cdot i_1 + c_2 \cdot j_1 + d_2 \cdot i_1j_1.\\
	\end{align*}
	Then there is an explicit isomorphism $$A \otimes B \cong \text{Mat}_4(K)$$
	given by
	$$i_1\otimes 1\mapsto M_{i_1} := \begin{bmatrix}
	0 & \alpha & 0 &0\\
	1 & 0 & 0 & 0\\
	0 & 0 & 0 & \alpha\\
	0 & 0 & 1 & 0\\
	
\end{bmatrix}$$

$$j_1  \otimes 1\mapsto M_{j_1} := \begin{bmatrix}
0 & 0 & \beta &0\\
0 & 0 & 0 & -\beta\\
1 & 0 & 0 & 0\\
0 & -1 & 0 & 0\\

\end{bmatrix}$$

$$1 \otimes i_2 \mapsto M_{i_2} := \begin{bmatrix}
a_1 & b_1 \cdot \alpha & c_1 \cdot \beta & - d_1 \cdot \alpha\beta\\
b_1 & a_1 & -d_1 \cdot \beta & c_1 \cdot \beta\\
c_1 & d_1 \cdot \alpha & a_1 & - b_1 \cdot \alpha\\
d_1 & c_1 & - b_1 & a_1\\

\end{bmatrix}$$

$$1 \otimes j_2 \mapsto M_{j_2} := \begin{bmatrix}
a_2 & b_2 \cdot \alpha & c_2 \cdot \beta & - d_2 \cdot \alpha\beta\\
b_2 & a_2 & -d_2 \cdot \beta & c_2 \cdot \beta\\
c_2 & d_2 \cdot \alpha & a_2 & - b_2 \cdot \alpha\\
d_2 & c_2 & - b_2 & a_2\\

\end{bmatrix}$$
\\
\end{lemma}
\begin{proof}

We have that $A \otimes A^{op}$ is isomorphic to a matrix algebra. More specifically, $A \otimes A^{op} \cong \text{End}_{K}(A)$ via $u \otimes v \mapsto (x \mapsto uxv)$, which makes $A \otimes A^{op} \cong \text{Mat}_4(K)$ after picking a basis for $A$. Hence,

$$
\begin{array}{ccc}
A \otimes B^{op} &\cong &\text{Mat}_4(K)\\
u \otimes v & \mapsto &(x \mapsto ux\psi(v)).\\
\end{array}
\;
$$
More explicitly, fixing the basis of $A$ to be $\{1, i_1, j_1, i_1j_1\}$, the isomorphism is as given in the statement of the  lemma.\\

\end{proof}

By taking $A = (c_P  a, c_Q  b), B= (c_R  c, c_S  d)$ in Lemma \ref{lem:  computation of M_P' etc}, we know that the matrices $M_P', M_Q', M_S', M_R'$ and  $M_{i_1}, M_{j_1}, M_{i_2}, M_{j_2}$ are equal up to conjugation by a matrix $C \in \text{GL}_4(K)$ via the Noether Skolem Theorem. After a change of coordinates for $\mathcal{K}_{\epsilon}\subset \pp^3$ according to $C$, we have that $M_P', M_Q', M_S', M_R'$ are equal to $M_{i_1}, M_{j_1}, M_{i_2}, M_{j_2}$. Lemma \ref{lem:  computation of M_P' etc} therefore reduces the problem of computing the $M_T'$ to that of computing an isomorphism between two quaternion algebras. See \cite[Corollary 4.2.3 ]{thesis} for the description of an explicit algorithm. Finally we solve for a matrix $A$ such that $M_T'=\lambda_TA^{-1}M_TA$ for some $\lambda_T \in \bar{K}$ and $T=P, Q, R, S$ by linear algebra.\\

\subsection{Explicit computation of $D$}\label{sec:explicit-computation-of-d}
In this section, we explain a method for computing the $K$-rational divisor $D$ on $J_{\epsilon}$ representing the divisor class $\phi_{\epsilon}^*(2\Theta)$. The idea is to compute it via the commutative diagram  \eqref{diagram: BS diagram} in  Remark \ref{rem: diagram of twited kummer}. \\

By Theorem \ref{theorem: explicit twist of J}, there is an explicit isomorphism $J_{\epsilon} \subset \pp^{15} \xrightarrow{\phi_{\epsilon}} J \subset \pp^{15}$. We write $u_0, ..., u_9, v_1, ..., v_6$ for the coordinates on the ambient space of $J_{\epsilon} \subset \pp^{15}$ and write $k_{11}, k_{12}, ..., k_{44}, b_1, ..., b_6$ for the coordinates on the ambient space of $J \subset \pp^{15}$. By the same theorem, $\phi_{\epsilon}$ is represented by a block diagonal matrix consisting of a block of size 10 corresponding to the even basis elements and a block of size 6 corresponding to the odd basis elements. Following Section \ref{sec:computing-the-twist-of--the-kummer-surface}, we can compute an explicit isomorphism  $\psi_{\epsilon} : \mathcal{K}_{\epsilon} \subset \pp^3 \rightarrow \mathcal{K} \subset \pp^3$ corresponding to $\epsilon$. We write $k_1', ..., k_4'$ for the coordinates on the ambient space of $\mathcal{K}_{\epsilon} \subset \pp^3$. Recall that since all points in $J[2]$ are defined over $K$, all coboundaries in $B^1(G_K, J[2])$ are trivial. So, we have that $\phi_{\epsilon}(\phi_{\epsilon}^{-1})^{\sigma}$ and $\psi_{\epsilon} (\psi_{\epsilon}^{-1})^{\sigma}$ both give the action of translation by some $\epsilon_{\sigma} \in J[2]$ such that $(\sigma \mapsto \epsilon_{\sigma})$  represents $\epsilon \in \text{Sel}^2(J)$. \\

Define $k_{ij}'=k_i'k_j'$. The isomorphism $\psi_{\epsilon}: \mathcal{K}_{\epsilon} \subset \pp_{k_i'}^3 \rightarrow \mathcal{K}\subset \pp_{k_i}^3$, induces a natural isomorphism $\tilde{\psi_{\epsilon}}: \pp^9_{k_{ij}'} \rightarrow \pp^9_{k_{ij}}$. More explicitly, suppose $\psi_{\epsilon}$ is represented by the $4 \times 4$ matrix $A$ where $(k_1':..., k_4') \mapsto (\sum_{i=1}^4A_{1i}k_i':...:\sum_{i=1}^4A_{4i}k_i')$. Then $\tilde{\psi_{\epsilon}}: \pp^9_{k_{ij}'} \rightarrow \pp^9_{k_{ij}}$ is given by $(k_{11}': k_{12}':... :k_{44}')\mapsto (\sum_{i, j=1}^4A_{1i}A_{1j}k_{ij}':\sum_{i, j=1}^4A_{1i}A_{2j}k_{ij}':...:\sum_{i, j=1}^4A_{4i}A_{4j}k_{ij}')$. \\

On the other hand, the isomorphism $\phi_{\epsilon}: J_{\epsilon} \subset \pp^{15}_{\{u_i, v_i\}} \rightarrow J \subset \pp^{15}_{\{k_{ij}, b_i\}}$ induces a natural  isomorphism $\tilde{\phi_{\epsilon}}: \pp^9_{u_i} \rightarrow \pp^9_{k_{ij}}$ represented by the $10\times 10$ block of the matrix representing $\phi_{\epsilon}$.  Since $\phi_{\epsilon}(\phi_{\epsilon}^{-1})^{\sigma}$ and $\psi_{\epsilon} (\psi_{\epsilon}^{-1})^{\sigma}$ both give the action of translation by some $\epsilon_{\sigma} \in J[2]$, we get  $\tilde{\phi_{\epsilon}}(\tilde{\phi_{\epsilon}}^{-1})^{\sigma}=\tilde{\psi_{\epsilon}}(\tilde{\phi_{\epsilon}}^{-1})^{\sigma}$. Therefore,  $\tilde{\psi_{\epsilon}}^{-1}\tilde{\phi_{\epsilon}}$ is defined over $K$ and we obtain the following commutative diagram that decomposes  the standard commutative diagram \eqref{diagram: BS diagram}:

\begin{equation}\label{eqn: D}
\begin{tikzcd}
J_{\epsilon} \subset \pp^{15}_{\{u_i, v_i\}} \arrow[r, "proj"] \arrow[d, "\phi_{\epsilon}"]& \pp^9_{u_i}  \arrow[rr, "(\tilde{\psi_{\epsilon}})^{-1}\tilde{\phi_{\epsilon}}"] \arrow[dr, "\tilde{\phi_{\epsilon}}"]&&\pp^9_{k_{ij}'} \arrow[dl, "\tilde{\psi_{\epsilon}}"] \arrow[r, "g_2"]&\mathcal{K}_{\epsilon} \subset \pp^3_{k_i'} \arrow[d, "\psi_{\epsilon}"]\\
J \subset \pp^{15}_{\{k_{ij}, b_i\}}\arrow[rr, "proj"]& & \pp^9_{k_{ij}} \arrow[rr, "g_1"]&&\mathcal{K} \subset \pp^3_{k_i},
\end{tikzcd}
\end{equation}
where $g_1:(k_{11}:...:k_{44}) \rightarrow (k_1: ...: k_4)$ and  $g_2:(k'_{11}:...:k'_{44}) \rightarrow (k'_1: ...: k'_4)$ are the projection maps. The composition of the morphisms on the bottom row gives the standard morphism $J \xrightarrow{|2\Theta|} \mathcal{K}\subset \pp^3$ and the composition of the morphisms on the top row  gives $J_{\epsilon} \xrightarrow{|\phi_{\epsilon}^*(2\Theta)|} \mathcal{K}_{\epsilon}\subset \pp^3$. \\

Let $D$ be the pull back on $J_{\epsilon}$ via $J_{\epsilon} \subset \pp^{15}_{\{u_i, v_i\}} \xrightarrow{proj} \pp^9_{u_i} \xrightarrow{(\tilde{\psi_{\epsilon}})^{-1}\tilde{\phi_{\epsilon}}} \pp^9_{k_{ij}'} \xrightarrow{proj} \pp^3_{k_i'}$ of the hyperplane section given  by $k_1'=0$.  This implies that $D$ is a $K$-rational divisor on $J_{\epsilon}$ representing the class of $\phi_{\epsilon}^*(2\Theta)$. Moreover, the pull back on $J_{\epsilon}$ via $J_{\epsilon} \subset \pp^{15}_{\{u_i, v_i\}} \xrightarrow{proj} \pp^9_{u_i} \xrightarrow{(\tilde{\psi_{\epsilon}})^{-1}\tilde{\phi_{\epsilon}}} \pp^9_{k_{ij}'}$ of the hyperplane section given  by $k_{11}'=0$ is $2D$. \\

\subsection{Explicit computation of $D_P, D_Q, D_R, D_S$}\label{sec:explicit-computation-of-dp-dq-dr-ds}

In this section, we explain  how to compute  the $K$-rational divisors $D_P, D_Q, D_R, D_S$ defined in Remark \ref{rem: DP, DQ, DR, DS}. More explicitly, for $T \in J[2]$, we give a method for computing a $K$-rational divisor  $D_T$ on $J_{\epsilon}$ representing the divisor class of $\phi_{\epsilon}^*(\tau_{T_1}^*(2\Theta))$ for some $T_1$ on $J$ such that $2T_1=T$. Recall that we assume all points in $J[2]$ are defined over $K$ and we have an explicit isomorphism $\phi_{\epsilon}: J_{\epsilon} \rightarrow J$ such that $(J_{\epsilon}, [2]\circ \phi_{\epsilon})$ is the 2-covering of $J$ corresponding to $\epsilon\in \text{Sel}^2(J)$. Recall $\delta: J(K) \rightarrow H^1(G_K, J[2])$ in \eqref{eqn: connecting map}. We first prove the following lemma.\\


\begin{lemma}\label{lem: longest vertical map over K}
Let $T \in J(K)$. Suppose $\phi_{\epsilon+\delta(T)}: J_{\epsilon+\delta(T)}\rightarrow J$ is an  isomorphism and $(J_{\epsilon+\delta(T)}, [2] \circ \phi_{\epsilon+\delta(T)})$ is the 2-covering of $J$ corresponding to $\epsilon+\delta(T) \in H^1(G_K, J[2])$. Let $T_1 \in J$ such that $2T_1=T$. Then, $\phi_{\epsilon+ \delta(T)}^{-1} \circ \tau_{T_1} \circ \phi_{\epsilon}: J_{\epsilon} \rightarrow J_{\epsilon+\delta(T)}$ is defined over $K$. 
\end{lemma}

\begin{proof}
Using the same argument as in the  proof of Lemma \ref{lem: existence of rational divisor D_T}(i), we know that $(J_{\epsilon}, [2] \circ \tau_{T_1} \circ \phi_{\epsilon})$ is the 2-covering of $J$ corresponding to $\epsilon+\delta(T) \in H^1(G_K, J[2])$. Since all points in  $J[2]$ are defined over $K$,  we have $\tau_{T_1} \circ \phi_{\epsilon} \circ ((\tau_{T_1} \circ \phi_{\epsilon} )^{-1})^{\sigma} = \phi_{\epsilon+ \delta(T)} \circ (\phi_{\epsilon+ \delta(T)}^{-1})^{\sigma}$, as required. \\
\end{proof}

Let $T \in J(K)$ with $2T_1=T$. Consider the commutative diagram below which is formed by two copies of the standard diagram \eqref{diagram: BS diagram}.  Note that all the horizontal maps are defined over $K$ as is the composition of the vertical map on the left by Lemma \ref{lem: longest vertical map over K}. Hence, the composition of the thick arrows is defined over $K$. Then the pull back on $J_{\epsilon}$ via the thick arrows of a hyperplane section on $\mathcal{K}_{\epsilon+\delta(T)} \subset \pp^3$ is a $K$-rational divisor $D_T$ on $J_{\epsilon}$ representing the divisor class $\phi_{\epsilon}^*(\tau_{T_1}^*(2\Theta))$. We note that in the case where $T \in J[2]$, the composition of the vertical maps on the left hand side of the the diagram below is in fact given by a $16 \times 16$ matrix defined over $K$ even though the individual maps are not defined over $K$.\\

\begin{equation}\label{diagram 1}
\begin{tikzcd}[column sep = large]
J_{\epsilon} \subset \pp^{15} \arrow[r, "|\phi_{\epsilon}^*(2\Theta)|"] \arrow[d, "\phi_{\epsilon}", very thick] & \mathcal{K}_{\epsilon} \subset \pp^3 \arrow[d, "\psi_{\epsilon}"]\\
J\subset \pp^{15}  \arrow[r, "|2\Theta|"] \arrow[d, "\tau_{T_1}", very thick] & \mathcal{K} \subset \pp^3\\
J\subset \pp^{15}  \arrow[r, "|2\Theta|"] \arrow[d, "\phi_{\epsilon+\delta(T)}^{-1}", very  thick]& \mathcal{K} \subset \pp^3\\
J_{\epsilon + \delta(T)} \subset \pp^{15} \arrow[r, "|\phi^*_{\epsilon+\delta(T)}(2\Theta)|", very thick] & \mathcal{K}_{\epsilon+\delta(T)} \subset \pp^3 \arrow[u, "\psi_{\epsilon + \delta(T)}"'].\\
\end{tikzcd}
\end{equation}

The bottom horizontal morphism $J_{\epsilon+\delta(T)} \xrightarrow{|\phi^*_{\epsilon+ \delta(T)}(2\Theta)|} \mathcal{K}_{\epsilon+\delta(T)} \subset \pp^3$ can be explicitly computed using the algorithm in Section \ref{sec:explicit-computation-of-d} with the Selmer element $\epsilon$ replaced by  $\epsilon+\delta(T)$. Also, by Theorem \ref{theorem: explicit twist of J}, we have explicit formulae for $\phi_{\epsilon}$ and  $\phi_{\epsilon+ \delta(T)}.$ Hence, to explicitly compute $D_T$,  we need to find a way to deal with $\tau_{T_1}$, for some $T_1$ such that $2T_1=T$.\\

Since we need to apply the above argument to $T=P, Q, R, S$, the basis for $J[2]$, it would suffice to compute the translation maps $\tau_{T_1}: J\subset \pp^{15} \rightarrow J \subset \pp^{15}$ when $T_1 \in J[4]$. This map is given by a $16 \times 16$ matrix. We show how to compute a $10 \times 16$ matrix in the following proposition. We then explain below why this is sufficient for our purposes.\\

\begin{proposition}\label{prop: partial translation by 4 torsion}
Suppose $T_1 \in J[4]$.  Given the coordinates of $T_1  \in J \subset \pp^{15}_{\{k_{ij}, b_i\}}$, we can compute the following composition of morphisms:

$$\Psi: J \subset \pp^{15}_{\{k_{ij}, b_i\}} \xrightarrow{\tau_{T_1}} J \subset \pp^{15}_{\{k_{ij}, b_i\}} \xrightarrow{proj} \pp^{9}_{k_{ij}}.$$\\

\end{proposition}
\begin{proof}
Let $T =2T_1 \in J[2]$. Recall that we let $M_T$ denote the action of translation by $T$ on $\mathcal{K} \subset \pp^3$. Then for any $P \in J$,  we have  $k_i(P+T)=\sum_{j=1}^4(M_T)_{ij}k_j(P)$ projectively as a vector of length 4, and  projectively as a vector of length 10, $k_{ij}(P+T_1)$ is equal to
$$k_i(P+T_1)k_j(P+T_1)=k_i(P+T_1) \sum_{l=1}^4(M_T)_{jl}k_l(P-T_1)=\sum_{l=1}^4(M_T)_{jl}k_l(P-T_1)k_i(P+T_1).$$
By \cite[Theorem 3.2]{the gp law paper}, there exists a $4 \times 4$ matrix of bilinear forms $\phi_{ij}(P, T_1)$, with explicit formula, that is projectively equal to the matrix $k_i(P-T_1)k_j(P+T_1)$. Since we have an explicit formula for $M_T$ in \cite[Chapter 3, Section 2]{the book}, we can partially compute the linear isomorphism $\tau_{T_1}$: $$\Psi: J \subset \pp^{15}_{\{k_{ij}, b_i\}} \xrightarrow{\tau_{T_1}} J \subset \pp^{15}_{\{k_{ij}, b_i\}} \xrightarrow{proj} \pp^{9}_{k_{ij}},$$
as required. \\
\end{proof}

\begin{remark}\label{rem: computing T_1}
Suppose $2T_1=T \in J[2]$. From the doubling formula on $\mathcal{K}$ as in \cite[Appendix~C]{the gp law paper}, we can compute the coordinates of the image of $T_1$ on $\mathcal{K} \subset \pp^3$ from the coordinates of the image of  $T$ on $\mathcal{K} \subset \pp^3$. This gives the 10 even coordinates,  $k_{ij}(T_1)$ and we can solve for the odd coordinates by the 72 defining equations of $J$ given in Theorem \ref{theorem: 72}. Note that by Lemma \ref{lem: longest vertical map over K}, we know the field of definition of $T_1$ is contained in the composition of the field of definition of $\phi_{\epsilon}$ and $\phi_{\epsilon+\delta(T)}$. Hence,  we can compute this field explicitly which helps solving for this point using MAGMA \cite{magma}.\\

\end{remark}

Consider  $T \in J[2]$ with $T_1 \in J[4]$ such that $2T_1=T$. We follow the discussion in Section \ref{sec:explicit-computation-of-d} with $\epsilon$ replaced by $\epsilon+\delta(T)$. This gives a diagram analogous to \eqref{eqn: D}. Let $k_{1, T}', ..., k_{4, T}'$ be the coordinates on the ambient space of $\mathcal{K}_{\epsilon+\delta(T)}\subset \pp^3$ and let $u_{0, T}, ..., u_{9, T}, v_{1, T}, ..., v_{6, T}$ be the coordinates on the ambient space of $J_{\epsilon+\delta(T)}\subset \pp^{15}$. Let $k_{ij, T}'=k_{i, T}'k_{j, T}'$. Decomposing the lower half of the diagram \eqref{diagram 1} gives the commutative diagram below:

\begin{equation}\label{equation: D_T}
\begin{tikzcd}[ampersand replacement=\&,column sep = large]
J_{\epsilon} \subset \pp^{15}_{\{u_i, v_i\}} \arrow[rr, "|\phi_{\epsilon}^*(2\Theta)|"] \arrow[d, "\phi_{\epsilon}", very thick]\&\& \mathcal{K}_{\epsilon} \subset \pp^3_{k_i'} \arrow[d, "\psi_{\epsilon}"]\\
J\subset \pp^{15}_{\{k_{ij}, b_i\}}  \arrow[rr, "|2\Theta|"] \arrow[d, "\tau_{T_1}"] \arrow[dr, "\Psi", very thick]\&\& \mathcal{K} \subset \pp_{k_i}^3\\
J\subset \pp^{15}_{\{k_{ij}, b_i\}}  \arrow[r, "proj"]\&\pp_{k_{ij}}^9 \arrow[r, "g_1"]\arrow[d, "(\tilde{\psi}_{\epsilon+\delta(T)})^{-1}", very thick]\& \mathcal{K} \subset \pp_{k_i}^3\\
J_{\epsilon + \delta(T)} \subset \pp^{15}_{\{u_{i, T}, v_{i, T}\}} \arrow[r] \arrow[u, "{\phi}_{\epsilon+\delta(T)}"]\&\pp_{k_{ij, T}'}^9 \arrow[r, "g_2", very thick]\& \mathcal{K}_{\epsilon+\delta(T)} \subset \pp_{k_{i, T}'}^3 \arrow[u, "\psi_{\epsilon + \delta(T)}"].\\
\end{tikzcd}
\end{equation}

Recall Proposition \ref{prop: partial translation by 4 torsion} explains how $\Psi$ can be explicitly computed and the composition of the thick arrows in \eqref{equation: D_T} is defined over $K$ by Lemma \ref{lem: longest vertical map over K}.  Let $D_T$ be the pull back on $J_{\epsilon}$ via the thick arrows in  \eqref{equation: D_T} of the hyperplane section given  by $k_{1, T}'=0$.  This implies that $D_T$ is a $K$-rational divisor on $J_{\epsilon}$ representing the class of $\phi_{\epsilon}^*(\tau_{T_1}^*(2\Theta))$. Moreover, the pull back on $J_{\epsilon}$ via $$J_{\epsilon} \subset \pp^{15}_{\{u_i, v_i\}} \xrightarrow{\phi_{\epsilon}} J \subset \pp^{15}_{\{k_{ij, b_i}\}}\xrightarrow{\Psi} \pp^9_{k{ij}} \xrightarrow{(\tilde{\psi}_{\epsilon+\delta(T)})^{-1}} \pp^9_{k_{ij, T}'}$$ of the hyperplane section given  by $k_{11, T}'=0$ is $2D_T$. \\

We now apply the above discussion with $T=P, Q, R, S$ and get that the  divisors $D_P, D_Q, D_R, D_S$ on $J_{\epsilon}$ described in Remark \ref{rem: DP, DQ, DR, DS} as required.\\

\begin{remark}\label{rem: explicit fp, fq, fr, fs}
From the above discussion and the discussion in Section \ref{sec:explicit-computation-of-d}, the $K$-rational functions $f_P, f_Q, f_R,  f_S$ in the formula for the Cassels-Tate pairing in Theorem \ref{thm: 1} are quotients of linear forms in the coordinates of the ambient space of $J_{\epsilon} \subset \pp^{15}$. They all  have the same denominator, this being the linear form that cuts out the divisor $2D$.\\

\end{remark}

\section{The Obstruction Map}
In this section, we will state and prove an explicit formula for the obstruction map $\text{Ob}: H^1(G_K, J[2]) \rightarrow \text{Br}(K)$. See below for the definition of this map. This  generalizes a formula in the elliptic curve case due to O'Neil \cite [Proposition 3.4]{oneil}, and later refined by Clark \cite[Theorem 6]{clark}. 
Although this is not needed for  the computation of the Cassels-Tate pairing, it explains why we needed to work with quaternion algebras in Section \ref{sec:computing-the-twist-of--the-kummer-surface}.\\


\begin{definition}
	The obstruction map $$\text{Ob}: H^1(G_{K}, J[2]) \rightarrow H^2(G_{K}, \bar{K}^*) \cong \text{Br}(K)$$ is the composition of the map $H^1(G_{K}, J[2]) \rightarrow H^1(G_{K}, \text{PGL}_4(\bar{K}))$ induced by the action of translation of $J[2]$ on $\mathcal{K} \subset \pp^3$, and the injective map $H^1(G_{K}, \text{PGL}_4(\bar{K})) \rightarrow H^2(G_{K}, \bar{K}^*)$ induced from the short exact sequence $0 \rightarrow \bar{K}^* \rightarrow \text{GL}_4(\bar{K}) \rightarrow \text{PGL}_4(\bar{K}) \rightarrow 0$.\\
	
\end{definition}
\begin{theorem}\label{thm: ob map}
	Let $J$ be the Jacobian variety of a genus two curve defined over a field $K$ with $char(K) \neq 2$. Suppose all points in $J[2]$ are defined over $K$. For $\epsilon \in H^1(G_K, J[2])$, represented by $(a, b, c, d) \in (K^*/(K^*)^2)^4$ as in Section \ref{sec:formula-for-the-cassels-tate-pairing},  the obstruction map $\mathrm{Ob}: H^1(G_K,J[2]) \rightarrow \mathrm{Br}(K)$ sends $\epsilon$ to the class of the tensor product of two quaternion algebras:
	$$\mathrm{Ob} (\epsilon) = (c_P  a, c_Q  b) + (c_Rc, c_Sd),$$
	where $c_P, c_Q, c_R, c_S \in K$ are such that $M_P^2 = c_PI, M_Q^2 = c_QI,M_R^2 = c_RI, \; \text{and} \; M_S^2 = c_SI$ as defined in Section \ref{sec:computing-the-twist-of--the-kummer-surface}.\\
	
\end{theorem}

\begin{proof}
	Let $N_P = 1/\sqrt{c_P}M_P, N_Q= 1/\sqrt{c_Q}M_Q, N_R=1/\sqrt{c_R}M_R, N_S =1/\sqrt{c_S}M_S \in \text{GL}_4(\bar{K})$. Then $N_P$ is a normalized representation in $\text{GL}_4(\bar{K})$ of $[M_P] \in \text{PGL}_4(K)$. Similar statements are true for $Q, R, S$. Notice that $N_P^2= N_Q^2 = N_R^2 = N_S^2= I$. So there is a uniform way of picking a representation in $\text{GL}_4(\bar{K})$ for the translation induced by $\alpha_1P + \alpha_2 Q + \alpha_3R + \alpha_4S$ for $\alpha_i \in \zz$, namely  $N_P^{\alpha_1} N_Q^{\alpha_2} N_R^{\alpha_3} N_S^{\alpha_4}.$\\
	
	Since  $\epsilon \in H^1(K, J[2]) $ is represented by $(a, b, c, d) \in (K^*/{K^*}^2)^4$ and $P, Q, R, S$ satisfy the Weil pairing matrix \eqref{equation: WP matrix}, a cocycle representation of $\epsilon$ is: $$\sigma \mapsto \tilde{b}_{\sigma}P +\tilde{a}_{\sigma}Q +\tilde{d}_{\sigma}R +\tilde{c}_{\sigma}S ,$$ where for each element $x \in K^*/(K^*)^2$, we define $\tilde{x}_{\sigma} \in \{0, 1\}$ such that $(-1)^{\tilde{x}_{\sigma}}=\sigma(\sqrt{x})/\sqrt{x}$.\\
	
	Now consider the following commutative diagram of cochains:
	
	$$
	\begin{tikzcd}
	C^1(G_K, \bar{K}^*) \arrow[r] \arrow[d, "d"]&  C^1(G_K, \text{GL}_4) \arrow[r] \arrow[d, "d"]& C^1(G_K, \text{PGL}_4) \arrow[d, "d"]\\
	C^2(G_K, \bar{K}^*) \arrow[r]&  C^2(G_K, \text{GL}_4) \arrow[r]& C^2(G_K, \text{PGL}_4).\\
	\end{tikzcd}
	$$
	
	Defining $N_{\sigma}=N_P^{\tilde{b}_{\sigma}}N_Q^{\tilde{a}_{\sigma}}N_R^{\tilde{d}_{\sigma}}N_S^{\tilde{c}_{\sigma}}$, we have
	$$
	\begin{array}{ccc}
	H^1(K, J[2])&\rightarrow &H^1(G_K, \text{PGL}_4)\\
	(a, b, c, d)& \mapsto & (\sigma \mapsto [N_{\sigma}]).
	\end{array}
	$$
	Then $(\sigma \mapsto [N_{ \sigma}]) \in C^1(G_K, \text{PGL}_4)$ lifts to $(\sigma \mapsto N_{ \sigma}) \in C^1(G_K, \text{GL}_4)$ which is then mapped to   $$((\sigma, \tau) \mapsto (N_{\tau})^{\sigma}N_{\sigma \tau}^{-1}N_{ \sigma}) \in C^2(G_K, \text{GL}_4).$$

	Note that $$N_P^{\sigma}= (\frac{1}{\sqrt{c_P}}M_P)^{\sigma}=\frac{1}{\sigma(\sqrt{c_P})}M_P =\frac{\sqrt{c_P}}{\sigma(\sqrt{c_P})}N_P= (-1)^{(\widetilde{c_P})_{\sigma}}N_P,$$ treating $c_P$ in $K^*/(K^*)^2$. Similar results also hold for $Q, R, S$. Observe that for any $x \in K^*/(K^*)^2$ and $ \sigma, \tau \in G_K$, we have $\tilde{x}_{\sigma}-\tilde{x}_{\sigma\tau}+ \tilde{x}_{\sigma}$ is equal to 0 or 2. Since $N_P^2= N_Q^2 = N_R^2 = N_S^2= I$, $[N_P, N_Q]=[N_R, N_S]=-I$ and the commutators of the other pairs are trivial, we have  \\
	\begin{align*}
	(N_{\tau})^{\sigma}N_{\sigma \tau}^{-1}N_{ \sigma}
	=& (N_P^{\tilde{b}_{\tau}}N_Q^{\tilde{a}_{\tau}}N_R^{\tilde{d}_{\tau}}N_S^{\tilde{c}_{\tau}})^{\sigma} \cdot N_S^{-\tilde{c}_{\sigma\tau}}N_R^{-\tilde{d}_{\sigma\tau}}N_Q^{-\tilde{a}_{\sigma\tau}}N_P^{-\tilde{b}_{\sigma\tau}} \cdot N_P^{\tilde{b}_{\sigma}}N_Q^{\tilde{a}_{\sigma}}N_R^{\tilde{d}_{\sigma}}N_S^{\tilde{c}_{\sigma}}\\
	=& (-1)^{(\widetilde{c_P})_{\sigma} \cdot \tilde{b}_{\tau}} \cdot (-1)^{(\widetilde{c_Q})_{\sigma} \cdot \tilde{a}_{\tau}} \cdot(-1)^{(\widetilde{c_R})_{\sigma} \cdot \tilde{d}_{\tau}} \cdot(-1)^{(\widetilde{c_S})_{\sigma} \cdot \tilde{c}_{\tau}}\\ 
	& \cdot N_P^{\tilde{b}_{\tau}}N_Q^{\tilde{a}_{\tau}} \cdot N_R^{\tilde{d}_{\tau}}N_S^{\tilde{c}_{\tau}}N_S^{-\tilde{c}_{\sigma\tau}}N_R^{-\tilde{d}_{\sigma\tau}} \cdot N_Q^{-\tilde{a}_{\sigma\tau}}N_P^{-\tilde{b}_{\sigma\tau}}N_P^{\tilde{b}_{\sigma}}N_Q^{\tilde{a}_{\sigma}} \cdot N_R^{\tilde{d}_{\sigma}}N_S^{\tilde{c}_{\sigma}}\\
	=& (-1)^{(\widetilde{c_P})_{\sigma} \cdot \tilde{b}_{\tau}} \cdot (-1)^{(\widetilde{c_Q})_{\sigma} \cdot \tilde{a}_{\tau}} \cdot(-1)^{(\widetilde{c_R})_{\sigma} \cdot \tilde{d}_{\tau}} \cdot(-1)^{(\widetilde{c_S})_{\sigma} \cdot \tilde{c}_{\tau}}\\ 
	& \cdot N_P^{\tilde{b}_{\tau}}N_Q^{\tilde{a}_{\tau}}N_Q^{-\tilde{a}_{\sigma\tau}}N_P^{-\tilde{b}_{\sigma\tau}}N_P^{\tilde{b}_{\sigma}}N_Q^{\tilde{a}_{\sigma}} \cdot N_R^{\tilde{d}_{\tau}}N_S^{\tilde{c}_{\tau}}N_S^{-\tilde{c}_{\sigma\tau}}N_R^{-\tilde{d}_{\sigma\tau}}N_R^{\tilde{d}_{\sigma}}N_S^{\tilde{c}_{\sigma}}\\
	=& (-1)^{(\widetilde{c_P})_{\sigma} \cdot \tilde{b}_{\tau}} \cdot (-1)^{(\widetilde{c_Q})_{\sigma} \cdot \tilde{a}_{\tau}} \cdot(-1)^{(\widetilde{c_R})_{\sigma} \cdot \tilde{d}_{\tau}} \cdot(-1)^{(\widetilde{c_S})_{\sigma} \cdot \tilde{c}_{\tau}} \cdot  (-1)^{ \tilde{a}_{\sigma} \cdot \tilde{b}_{\tau}} \cdot  (-1)^{\tilde{c}_{\sigma} \cdot \tilde{d}_{\tau} } \cdot I.\\
	\end{align*}

	On the other hand, $(c_P, c_Q) \otimes (c_R, c_S)$ is isomorphic to  $\langle M_P, M_Q, M_R, M_S\rangle = \text{Mat}_4(K)$ which represents the identity element in the Brauer group. Hence,  we have $$(c_P  a, c_Q  b) + (c_R  c, c_S d)= (a, b) + (c, d)+ (c_P, b)+(c_Q, a) +(c_R, d)+(c_S, c),$$
	which is precisely represented by a cocycle that sends  $(\sigma, \tau)$ to \\
	$$(-1)^{(\widetilde{c_P})_{\sigma} \cdot \tilde{b}_{\tau}} \cdot (-1)^{(\widetilde{c_Q})_{\sigma} \cdot \tilde{a}_{\tau}} \cdot(-1)^{(\widetilde{c_R})_{\sigma} \cdot \tilde{d}_{\tau}} \cdot(-1)^{(\widetilde{c_S})_{\sigma} \cdot \tilde{c}_{\tau}} \cdot  (-1)^{ \tilde{a}_{\sigma} \cdot \tilde{b}_{\tau}} \cdot  (-1)^{\tilde{c}_{\sigma} \cdot \tilde{d}_{\tau} },$$
	for all $\sigma, \tau \in G_K$ as required.

\end{proof}

\section{Bounding the Set of Primes}\label{sec:prime-bound}

In this section, we directly show that the formula for $\langle \epsilon, \eta \rangle_{CT}$ in Theorem \ref{thm: 1} is actually always a finite product,  as mentioned in Remark \ref{rem: CTP finite product}. Since for a local field with odd residue characteristic, the Hilbert symbol between $x$ and $y$ is trivial when the valuations of $x$, $y$ are both 0, it suffices to find a finite set $S$ of places of $K$, such that outside $S$ the first arguments of the Hilbert symbols in the formula for $\langle \epsilon, \eta \rangle_{CT}$  have valuation 0 for some choice of the local point $P_v$.\\

Let $\mathcal{O}_K$ be the ring of integers for the number field $K$. By rescaling the variables, we assume the genus two curve is defined by $y^2=f(x)=f_6x^6 + ...+f_0$ where the $f_i$ are in  $\mathcal{O}_K$.\\

The first arguments of the Hilbert symbols in the  formula for $\langle \epsilon, \eta \rangle_{CT}$ are $f_P(P_v)$,  $f_Q(P_v)$, $f_R(P_v)$ or $f_S(P_v),$  where $f_P, f_Q, f_R, f_S$ can be computed as the quotients of two linear forms in $\pp^{15} $ with the denominators being the same, as explained in  Remark \ref{rem: explicit fp, fq, fr, fs}. Since we know that the Cassels-Tate pairing is independent of the choice of the local points $P_v$  as long as these are chosen to avoid all the zeros and poles, it suffices to make sure that there exists at least one local point $P_v$ on $J_{\epsilon}$ for which the values of the quotients of the linear forms  all have valuation 0 for all $v$ outside  $S$. The idea is to first  reduce the problem to the residue field. \\

 By Theorem \ref{theorem: explicit twist of J} and Remark \ref{rem: explicit twist of J formula},  we have an explicit formula for the linear isomorphism 
$$J_{\epsilon}\subset \pp^{15} \xrightarrow{\phi_{\epsilon}} J \subset \pp^{15},$$
which is defined over $K'=K(\sqrt{a}, \sqrt{b}, \sqrt{c}, \sqrt{d})$ where  $\epsilon = (a, b, c, d) \in (K^*/(K^*)^2)^4$. Suppose $\phi_{\epsilon}$ is represented by $M_{\epsilon} \in \text{GL}_{16}(K')$. Note we can assume that all entries of $M_{\epsilon}$ are in $\mathcal{O}_{K'}$, the ring of integers of $K'$.\\

\begin{notation}\label{notation: reduction}
	Let $K$ be a local field with valuation ring $\mathcal{O}_K$, uniformizer $\pi$ and residue field $k$. Let $X \subset \pp^N$ be a variety defined over $K$ and  $I(X) \subset K[x_0,...,x_N]$ be the ideal of $X$.  Then the reduction of $X$, denoted by $\bar{X}$,  is the variety defined by the polynomials $\{\bar{f} : f \in I(X) \cap \mathcal{O}_K[x_0,...,x_N]\}$. Here $\bar{f}$ is the polynomial obtained by reducing all the coefficients of $f$ modulo $\pi$. Note that this definition of the reduction of a variety $X \subset \pp^N$ defined over a local field $K$ is equivalent to taking the special fibre of  the closure of $X$ in $\pp^{N}_S$, where
	$S = \operatorname{Spec} \mathcal {O}_K$.\\
\end{notation}

Let   $S_0=\{\text{places of bad reduction for } \mathcal{C}\} \cup \{\text{places dividing 2}\} \cup \{\text{infinite places}\}$. Fix a place  $v \notin S_0$ and suppose it is above the prime $p$. We now treat $J, J_{\epsilon}$ and $\mathcal{C}$ as varieties defined over the local field $K_v$. Let $\mathcal{O}_v$ denote the valuation ring of $K_v$ and $\mathbb{F}_{q}$ denote its residue field,  where  $q$ is some power of $p$. It can be shown that $\bar{J}$ is also an abelian variety as the defining equations of  $J$ are defined over $\mathcal{O}_v$ and are derived algebraically in terms of the coefficients of the defining equation of the genus two curve $\mathcal{C}$ by Theorem \ref{theorem: 72}. In fact,  $\bar{J}$ is the Jacobian variety of  $\bar{\mathcal{C}}$, the reduction of $\mathcal{C}$. \\

Now fix a place $v'$ of $K'$ above the place $v$ of $K$. Let $\mathcal{O}_{v'}$ and $\mathbb{F}_{q^r}$ denote the valuation ring and the residue field of $K'_{v'}$. It can be checked that  as long as $v'$ does not divide $\det M_{\epsilon} \in \mathcal{O}_{K'}$, the following diagram commutes and $\bar{M_{\epsilon}}$ is a well defined linear isomorphism defined over the residue field $\mathbb{F}_{q^r}$ between two varieties defined over $\mathbb{F}_q$: \\


\[
\begin{tikzcd}
J_{\epsilon} \subset \pp^{15}\arrow[r, "M_{\epsilon}"] \arrow[d, "\text{reduction}"]& J \subset \pp^{15} \arrow[d, "\text{reduction}"]\\
\bar{J_{\epsilon}} \subset \pp^{15} \arrow[r, "\bar{M_{\epsilon}}"]& \bar{J}\subset \pp^{15},\\
\end{tikzcd}
\]
where $\bar{M_{\epsilon}}$ denotes the reduction of the matrix $M_{\epsilon}$ over the residue field $\mathbb{F}_{q^r}$.\\

This linear isomorphism $\bar{M_{\epsilon}}$ implies that $\bar{J_{\epsilon}}$ is smooth whenever $\bar{J}$ is. In this case,  $\bar{J_{\epsilon}}$ is a twist of $\bar{J}$ and it in fact a 2-covering of $\bar{J}$. Indeed, the surjectivity of the natural map $\text{Gal}(K'_{v'}/K_v) \rightarrow \text{Gal}(\mathbb{F}_{q^r}/\mathbb{F}_q)$ shows that  $M_{\epsilon}(M_{\epsilon}^{-1})^{\sigma}=\tau_{P_{\sigma}}$ for all $\sigma \in \text{Gal}(K'_{v'}/K_v)$ implies that  $\bar{M_{\epsilon}}(\bar{M_{\epsilon}}^{-1})^{\bar{\sigma}}=\tau_{\bar{P_{\bar{\sigma}}}}$ for all $\bar{\sigma} \in \text{Gal}(\mathbb{F}_{q^r}/\mathbb{F}_q)$.
We know any principal  homogeneous space of $\bar{J}$ over a finite field has a point by \cite[Theorem 2]{lang} and so is trivial by Proposition \ref{prop:2 covering has a point}. Therefore,  there exists an isomorphism  $\bar{J_{\epsilon}} \xrightarrow{\psi} \bar{J}$ defined over $\mathbb{F}_q$.  Hence, as long as $v \notin S_0$ and $v$ does not divide  $N_{K'/K} (\det M_{\epsilon})$, $\bar{J_{\epsilon}}$ has the same number of $\mathbb{F}_q$-points as $\bar{J}$. By the Hasse-Weil bound, we know the number of $\mathbb{F}_q$-points on $\mathcal{C}$ is bounded below by $q-1-4\sqrt{q}$. Since we can represent points on $\bar{J}$ by  pairs of points on $\bar{\mathcal{C}}$ and this representation is unique other than the identity point on $\bar{J}$. The number of $\mathbb{F}_q$-points on $\bar{J}$ is bounded below by $(q-1-4\sqrt{q})(q-3-4\sqrt{q})/2$.\\

On the other hand, let $l_1, ..., l_5$ be the 5 linear forms that appear as numerator or denominator of $f_P, f_Q, f_R, f_S$. We can assume that the coefficients of $l_i$ are in $\mathcal{O}_K$ by scaling, for all $i = 1, ..., 5$. Fix a place $v$ of $K$ that does not divide all the coefficients of $l_i$, for any $i=1, ..., 5$. Let $H_i$ be the hyperplane defined by the linear form $l_i$ and  $\bar{H_i}$ be its reduction, which is a hyperplane defined over the residue field $\mathbb{F}_q$,  We need to bound the number of $\mathbb{F}_q$-points of $\bar{J_{\epsilon}}$ that lie on one of the hyperplanes $\bar{H_i}$. Let $r_i$ be the number of irreducible components of $\bar{J_{\epsilon}} \cap \bar{H_i}$. By \cite[Chapter~1, Theorem 7.2 (Projective Dimension Theorem) and Theorem 7.7]{hartshorne}, we know that each irreducible component $C^i_j$ of $\bar{J_{\epsilon}} \cap \bar{H_i}$, where $j =1, ..., r_i$, is a curve and the sum of degrees of all the irreducible components counting intersection multiplicity is $\deg \bar{J_{\epsilon}}=32.$ Leting $d^i_j =\deg C^i_j$,  we have  $\sum_{j=1}^{r_i}d^i_j \le 32$ for all $i$.

\begin{lemma}
	Let $C \subset \pp^N$ be a curve of degree $d$. Then 
	$\# C(\mathbb{F}_q)\le d(q+1)$.
	\end{lemma}
\begin{proof}
	We may assume that $C$ is contained in no hyperplane. Then projection to the first two coordinates gives a nonconstant morphism $C \rightarrow \pp^1$ of degree $\le d$. Since $\#\pp^1(\mathbb{F_q})=q+1$, this gives the required bound.\\
	\end{proof}

By applying the above lemma to each $C_j^i$, we get the number of $\mathbb{F}_q$-points of  $\bar{J_{\epsilon}}$ that lie on one of the hyperplanes $\bar{H_i}, i = 1, ..., 5,$ is no more than $$\sum_{i=1}^{5} \sum_{j=1}^{r_i} d^i_j  \cdot (q+1) \le 160(q+1).$$



We compute that for any $x > 500$, we have  $(x-1-4\sqrt{x})(x-3-4\sqrt{x})/2>160(x+1)$. Recall $q$ is a power of $p$. Hence, if $v$ is a place of $K$ above the prime $p>500$ such that $v \notin S_0$ and $v$ does not divide $N_{K'/K}(\det M_{\epsilon})$ or all the coefficients of $l_i$ for some $i$, we have a smooth $\mathbb{F}_q$-point  on $\bar{J_{\epsilon}}$ which by Hensel's Lemma \cite[Exercise C.9(c)]{hensel} lifts to the point $P_v$ as required. This implies that the first arguments of the Hilbert symbols in the formula for the local Cassels-Tate pairing of $\langle \epsilon, \eta \rangle_{CT}$  have valuation 0. It can be checked that since $v \notin S_0$, the second arguments of these Hilbert symbols also  have valuation 0. Hence, the formula for the Cassels-Tate is indeed always a finite product.\\

Note that in the case where $K=\Q$ or more generally if $K$ has class number 1, we can always make the linear forms primitive by scaling. Therefore, in this case, the subset $\{$places dividing all the coefficients of the denominator or the numerator of $f_P, f_Q, f_R \text{ or } f_S\}$ is empty.\\

\section{Worked Example}
Now we demonstrate the algorithm with a worked example computed using MAGMA \cite{magma}. In particular, we will see with this example, that computing the Cassels-Tate pairing on $\text{Sel}^2(J)$ does improve the rank bound obtained via a 2-descent. This genus two curve was kindly provided by my PhD supervisor, Tom Fisher, along with a list of other genus two curves for me to test the algorithm.\\

Consider the following genus two curve
$$\mathcal{C}: y^2=-10x(x+10)(x+5)(x-10)(x-5)(x-1).$$
Its Jacobian variety $J$ has all its two torsion points defined over $\Q$. A set of generators of $J[2]$ compatible with the Weil pairing matrix \eqref{equation: WP matrix} are $P= \{(0, 0), (-10, 0)\}, Q= \{(0, 0), (-5, 0)\},
R= \{(10, 0), (5, 0)\}, S= \{(10, 0), (1, 0)\}.$ We identify $H^1(G_K, J[2])=(\Q^*/(\Q^*)^2)^4$ as in Section \ref{sec:formula-for-the-cassels-tate-pairing}. Consider $\epsilon, \eta \in \text{Sel}^2(J)$ represented by $(-33, 1, -1, -11)$ and $( 11, 1, -1, -11)$ respectively. The images of $[P], [Q], [R], [S]$ via $\delta: J(\Q)/2J(\Q) \rightarrow H^1(G_{\Q}, J[2])$, computed via the explicit formula as in \cite[Chapter 6, Section 1]{the book},  are $\delta([P])=(-66, 1, 6, 22), \delta([Q])=(-1, 1, 3, 1),
	\delta([R])=( 6, 3, 1, 3 ), \delta([S])=(22, 1, -3, -11 ).$ Now following the discussions in Sections \ref{sec:explicit-computation-of-d} and \ref{sec:explicit-computation-of-dp-dq-dr-ds}, we can compute, using the coordinates $c_0, ..., c_9, d_1, ..., d_6$ for $J_{\epsilon} \in \pp^{15}$ as described  in Remark \ref{rem: explicit twist of J formula}. we have 
	\begin{align*}
	k_{11}'&= 618874080c_0 - 496218440c_1 - 390547052c_3
	+ 205551080c_4\\
	&+ 384569291c_6 + 52868640c_8;\\
	\\
	k_{11, P}'&= -36051078800000c_2 + 8111492730000c_3 + 265237150000c_7\\
	& - 196928587500c_8 - 6786529337500c_9 + 22531924250d_2 \\
	&- 126449158891d_4 - 117221870375d_5 + 937774963000d_6; \\
	\\
	k_{11, Q}'&= 134800c_1 + 235600c_3 + 62000c_4 + 52235c_6 + 60016d_1 - 5456d_5;\\
	\\
	k_{11, R}'&= -30223125c_6 + 4050000c_8 - 49750d_3 + 709236d_4 \\
	\\
	k_{11, S}' &= 4724524800c_1 + 8557722360c_3 + 13102732800c_4 + 1258642935c_6 \\
	&+ 7291944000c_9 -
	2709362304d_1 + 97246845d_2 + 8475710d_3\\
	& + 30788208d_5.
	\end{align*}
	
	Hence, we have explicit formulae for 
	
	$$f_P = \frac{k_{11, P}'}{k_{11}'}, f_Q = \frac{k_{11, Q}'}{k_{11}'}, f_R = \frac{k_{11, R}'}{k_{11}'},  f_S = \frac{k_{11, S}'}{k_{11}'}.$$
	
	In particular, they are defined over $\Q$ as claimed. From Section \ref{sec:prime-bound}, we compute that only primes below 500 can potentially contribute to $\langle \epsilon, \eta \rangle_{CT}$. Then, it turns out that the only nontrivial local Cassels-Tate pairings between $\epsilon$ and $\eta$ are at places $11, 19, \infty$ and  $\langle \epsilon, \eta \rangle_{CT}=-1$.\\

Under the isomorphism $H^1(G_{\Q}, J[2]) \rightarrow (\Q^*/(\Q^*)^2)^4$, $\text{Sel}^2(J)$ has size $2^6$ and is generated by $(-33, 1, -1, -11), (11, 1, -1, -11 ), (66, 1, 2, 22),(11, 1, 2, 22), (3, 3, 3, 3), (3, 1, 3, 1).$ Since $\mathcal{C}$ has rational points, the Cassels-Tate pairing can be shown to be  alternating using \cite[Corollary 7]{poonen stoll}. Since all the two torsion points on $J$ are rational and $\langle \epsilon, \eta \rangle_{CT}=-1$, we get $|\ker \langle \;, \; \rangle_{CT}|=2^4$.\\

Indeed, we verified that the Cassels-Tate pairing matrix, with the generators of $\text{Sel}^2(J)$ listed above, is
$$\begin{bmatrix}
1 &-1& 1&  1&  -1& -1\\
-1& 1& 1&  -1& 1& -1\\
1 &1 &1 & 1 & 1 & 1\\
1  &-1 &1& 1& -1& -1\\
-1 & 1 & 1 & -1 & 1 &-1\\
-1 &-1 & 1 &-1 &-1 & 1\\
\end{bmatrix},$$
which is a rank 2 matrix.\\

As shown in \cite[Remark 1.9.4(ii)]{thesis}, in the case where all points in $J[2]$ are defined over the base field, computing the Cassels-Tate pairing on $\text{Sel}^2(J)$ gives the same rank bound as obtained from carrying out a $4$-descent, i.e. computing $\text{Sel}^4(J)$, which can potentially give a better rank bound  than the one given by a 2-descent. In this example, the rank bound coming from 2-decent was $\rank(J(\Q))\le 2$. Our calculations of the Cassels-Tate pairing on $\text{Sel}^2(J)$ improves this bound and in fact shows that $\rank(J(\Q))=0$.\\


\end{document}